\PassOptionsToPackage{unicode=true}{hyperref} 
\PassOptionsToPackage{hyphens}{url}
\documentclass[11pt,a4paper]{article}
\usepackage{lmodern}
\usepackage{amssymb,amsmath}
\usepackage{ifxetex,ifluatex}
\usepackage{array}
\usepackage[T1]{fontenc}
\usepackage[utf8]{inputenc}
\usepackage[]{microtype}
\UseMicrotypeSet[protrusion]{basicmath} 
\usepackage{xcolor}

\usepackage{hyperref}
\hypersetup{
  pdftitle={A species approach to the twelvefold way},
  pdfauthor={Anders Claesson},
  pdfborder={0 0 0},
  breaklinks=true}
\usepackage[top=3.1cm, bottom=3.1cm, left=3.17cm, right=3.17cm]{geometry}
\providecommand{\tightlist}{%
  \setlength{\itemsep}{0pt}\setlength{\parskip}{0pt}}

\makeatletter 
\renewcommand{\@seccntformat}[1]{\csname the#1\endcsname.\;\;}
\makeatother

\ifx\paragraph\undefined\else
\let\oldparagraph\paragraph
\renewcommand{\paragraph}[1]{\oldparagraph{#1}\mbox{}}
\fi
\ifx\subparagraph\undefined\else
\let\oldsubparagraph\subparagraph
\renewcommand{\subparagraph}[1]{\oldsubparagraph{#1}\mbox{}}
\fi

\makeatletter
\def\fps@figure{htbp}
\makeatother

\usepackage{concrete}
\usepackage{tikz}
\usetikzlibrary{matrix,arrows,patterns,calc,trees}
\usetikzlibrary{decorations.pathmorphing}
\usetikzlibrary{decorations.markings}
\tikzset{zigzag/.style={gray,decoration={zigzag,segment length=5.5,amplitude=1.1},line join=round,decorate}}
\usepackage{amssymb,amsmath,amsthm}

\newcommand{\ZZ}{\mathbb{Z}}

\newcommand{\A}{\mathcal{A}}
\newcommand{\B}{\mathcal{B}}
\renewcommand{\C}{\mathcal{C}}
\newcommand{\D}{\mathcal{D}}
\renewcommand{\G}{\mathcal{G}}
\newcommand{\typ}[1]{\widetilde{#1}}
\newcommand{\typeX}[1]{[\hspace{0.5pt}#1\hspace{0.5pt}]_X}
\newcommand{\typeY}[1]{[\hspace{0.5pt}#1\hspace{0.5pt}]_Y}

\newcommand{\simsum}{{\textstyle\mathop{\widetilde{\sum}}}}
\newcommand{\Fun}[1][R]{\Phi_{#1}}
\newcommand{\uFun}[1][R]{\typ{\Phi}_{#1}}
\newcommand{\tXFun}[1][R]{\simsum_U\Phi_{#1}[U,Y)}
\newcommand{\xFun}[1][R]{\simsum_U\Phi_{#1}[U,y)}
\newcommand{\tYFun}[1][R]{\simsum_V\hspace{0.8pt}\Phi_{#1}(X,V]}
\newcommand{\yFun}[1][R]{\simsum_V\hspace{0.8pt}\Phi_{#1}(x,V]}

\newcommand{\uNonEmptyE}{\typ{E}_+}

\newcommand{\Pow}{\mathcal{P}}
\newcommand{\Der}{\mathrm{Der}}

\newcommand{\Sym}{\mathcal{S}}

\newcommand{\Cyc}{\mathcal{C}}

\newcommand{\Bal}{\mathrm{Bal}}

\newcommand{\Par}{\mathrm{Par}}
\newcommand{\uPar}{\typ{\Par}}
\newcommand{\Lah}{\mathrm{Lah}}

\DeclareMathOperator{\cyc}{cyc}
\DeclareMathOperator{\Fix}{Fix}
\newcommand{\Id}{\mathrm{Id}}
\newcommand{\Set}{\mathrm{Set}}
\newcommand{\FinSet}{\mathrm{FinSet}}
\newcommand{\CoreFinSet}{\mathbb{B}}
\newcommand{\eps}{\epsilon}
\newcommand{\vphi}{\varphi}

\newtheorem{theorem}{Theorem}
\newtheorem{lemma}[theorem]{Lemma}
\newtheorem{proposition}[theorem]{Proposition}
\newtheorem{corollary}[theorem]{Corollary}

\theoremstyle{definition}

\title{A species approach to Rota's twelvefold way}
\author{Anders Claesson}
\date{7 September 2019}

\begin{document}
\maketitle

\thispagestyle{empty}

\begin{abstract}
  An introduction to Joyal's theory of combinatorial species is given
  and through it an alternative view of Rota's twelvefold way emerges.
\end{abstract}

\hypertarget{introduction}{%
\section{Introduction}\label{introduction}}

\begin{minipage}{0.135\textwidth}
\hspace{0.01\textwidth}\includegraphics[width=0.84\textwidth]{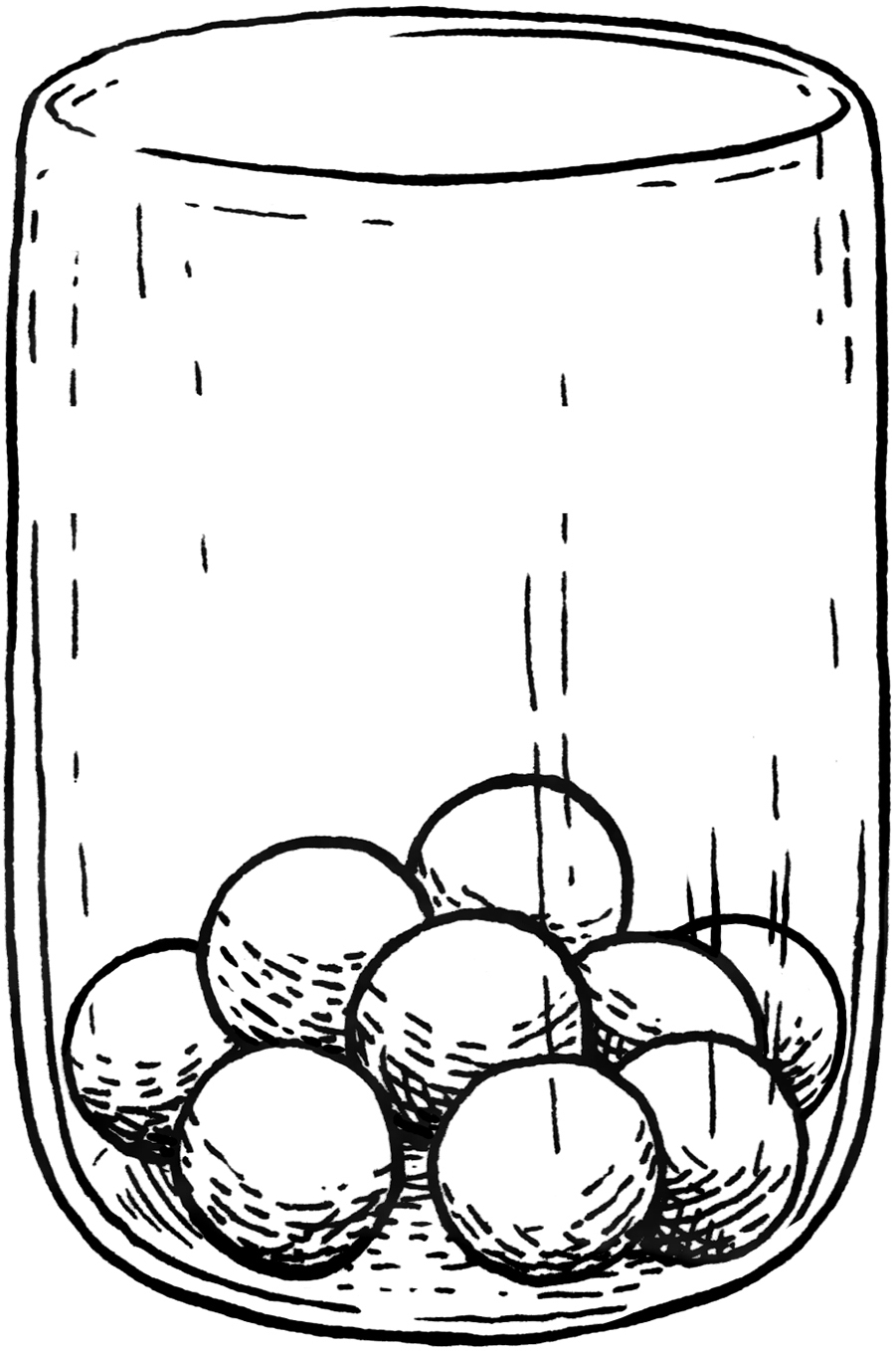}\hspace{0.15\textwidth}
\end{minipage}
\begin{minipage}{0.86\textwidth}
In how many ways can $n$ balls be distributed into $k$ urns? If there
are no restrictions given, then each of the $n$ balls can be freely
placed into any of the $k$ urns and so the answer is clearly $k^n$.
But what if we think of the balls, or the urns, as being identical
rather than distinct?
What if we have to place at least one ball, or can
place at most one ball, into each urn?
\end{minipage}\vspace{3pt}

The twelve cases resulting from
exhaustively considering these options are collectively referred to as
the \emph{twelvefold way}, an account of which can be found in Section~1.4 of
Richard Stanley's~\cite{EC1} excellent {\it Enumerative Combinatorics}, Volume~1. He
attributes the idea of the twelvefold way to Gian-Carlo Rota and its name to
Joel Spencer.

Stanley presents the twelvefold way in terms of counting functions,
$f:U\to V$, between two finite sets. If we think of $U$ as a set of
balls and $V$ as a set of urns, then requiring that each urn contains
at least one ball is the same as requiring that $f$ is surjective, and
requiring that each urn contain at most one ball is the same as
requiring that $f$ is injective. To say that the balls are identical, or
that the urns are identical, is to consider the equality of such
functions up to a permutation of the elements of $U$, or up to a
permutation of the elements of $V$.

We provide a uniform treatment of the twelvefold way, as well as some
extensions of it, using Joyal's theory of combinatorial species, and in
doing so we will endeavor to give an accessible introduction to this
beautiful theory. Section~\ref{combinatorial-species} contains this
introduction to species and it is based on two articles of
Joyal~\cite{Joyal1981, Joyal1986} and the book {\it Combinatorial species and
tree-like structures} by Bergeron, Labelle, and
Leroux~\cite{Bergeron1998}. In Section~\ref{the-twelvefold-way} we give
the species interpretation of the twelvefold way.

\hypertarget{combinatorial-species}{%
\section{Combinatorial species}\label{combinatorial-species}}

Before defining what a
combinatorial species is we shall consider simple graphs as a motivating
example.

A simple undirected (labeled) graph is a pair $(U,E)$, where $U$ is
a set of vertices and $E$ is a set of edges, which are 2-element
subsets of $U$. For instance, the graph with vertex set
$U = \bigl\{a,b,c,d,e\bigr\}$ and edge set
$E = \bigl\{ \{c,b\}, \{b,d\}, \{d,a\}, \{a,e\}, \{d,e\}\bigr\}$ is
depicted here:
\[
  \begin{tikzpicture}[semithick, scale=0.9]
    \begin{scope}[every node/.style={circle,draw,fill=black,minimum size=3.5pt,inner sep=0pt}]
      \node [label={[yshift=2pt]\footnotesize $c$}] (A) at (0,0.5) {};
      \node [label={[yshift=2pt]\footnotesize $b$}] (B) at (1,0.5) {};
      \node [label={[yshift=2pt]\footnotesize $d$}] (C) at (2,0.5) {};
      \node [label={[xshift=8pt, yshift=-7pt]\footnotesize $e$}] (D) at (3,0) {};
      \node [label={[xshift=8pt, yshift=-6pt]\footnotesize $a$}] (E) at (3,1) {};
      \draw (A) -- (B) -- (C) -- (D) -- (E) -- (C);
    \end{scope}
  \end{tikzpicture}
\]
Informally, a graph in which vertices are indistinguishable is said
to be unlabeled:
\[
  \begin{tikzpicture}[semithick, scale=0.9]
    \begin{scope}[every node/.style={circle,draw,fill=black,minimum size=3.5pt,inner sep=0pt}]
      \node [label={[text=white, yshift=2pt]\footnotesize $c$}] (A) at (0,0.5) {};
      \node [label={[text=white, yshift=2pt]\footnotesize $b$}] (B) at (1,0.5) {};
      \node [label={[text=white, yshift=2pt]\footnotesize $d$}] (C) at (2,0.5) {};
      \node [label={[text=white, xshift=8pt, yshift=-7pt]\footnotesize $e$}] (D) at (3,0) {};
      \node [label={[text=white, xshift=8pt, yshift=-6pt]\footnotesize $a$}] (E) at (3,1) {};
      \draw (A) -- (B) -- (C) -- (D) -- (E) -- (C);
    \end{scope}
  \end{tikzpicture}
\]
Formally, an \emph{isomorphism} of graphs $(U, E)$ and $(V,E')$
is a bijection $\sigma:U\to V$ that preserves the adjacency relation:
$\{x,y\}\in E$ if and only if $\{\sigma(x), \sigma(y)\}\in E'$.
Then an \emph{unlabeled graph} is an isomorphism class.

A species defines not only a class of (labeled) combinatorial objects
but also how those objects are affected by relabeling. This mechanism is
called \emph{transport of structure}. In particular, a species carries the
necessary information for defining the unlabeled counterpart of itself.

For graphs one defines $\G[U]$ as the set of graphs on the finite set
$U$, and for a bijection $\sigma: U\to V$ one defines the transport
of structure, $\G[\sigma]:\G[U]\to \G[V]$, by
$\G[\sigma](U,E) = \big(V,\sigma(E)\big)$, where
$\sigma(E) = \bigl\{\{\sigma(x),\sigma(y)\}: \{x,y\} \in E \bigr\}$.
For instance, if
$\sigma = \left(\begin{smallmatrix}a & b & c & d & e \\ 3 & 1 & 4 & 5 & 2 \end{smallmatrix}\right)$
is the bijection from $U=\{a,b,c,d\}$ to $V=\{1,2,3,4\}$ mapping $a$ to $3$, $b$ to $1$, etc.,
then
\[
  \begin{tikzpicture}[scale=0.95]
    \begin{scope}[semithick, every node/.style={circle,draw,fill=black,minimum size=3.5pt,inner sep=0pt}]
      \node [label={[yshift=2pt]\footnotesize $c$}] (A) at (0,0.5) {};
      \node [label={[yshift=2pt]\footnotesize $b$}] (B) at (1,0.5) {};
      \node [label={[yshift=2pt]\footnotesize $d$}] (C) at (2,0.5) {};
      \node [label={[xshift=8pt, yshift=-7pt]\footnotesize $e$}] (D) at (3,0) {};
      \node [label={[xshift=8pt, yshift=-6pt]\footnotesize $a$}] (E) at (3,1) {};
      \draw (A) -- (B) -- (C) -- (D) -- (E) -- (C);
    \end{scope}
    \begin{scope}[semithick, xshift=30,yshift=-45,every node/.style={circle,draw,fill=black,minimum size=3.5pt,inner sep=0pt}]
      \node [label={[yshift=-16pt]\footnotesize $4$}] (a) at (0,0.5) {};
      \node [label={[yshift=-16pt]\footnotesize $1$}] (b) at (1,0.5) {};
      \node [label={[yshift=-16pt]\footnotesize $5$}] (c) at (2,0.5) {};
      \node [label={[xshift=8pt, yshift=-7pt]\footnotesize $2$}] (d) at (3,0) {};
      \node [label={[xshift=8pt, yshift=-6pt]\footnotesize $3$}] (e) at (3,1) {};
      \draw (a) -- (b) -- (c) -- (d) -- (e) -- (c);
    \end{scope}
    \begin{scope}[thin, decoration={
        markings,
        mark=at position 0.5 with {\arrow{>}}}
      ]
      \draw[dashed, postaction={decorate}] (A) -- (a);
      \draw[dashed, postaction={decorate}] (B) -- (b);
      \draw[dashed, postaction={decorate}] (C) -- (c);
      \draw[dashed, postaction={decorate}] (E) -- (e);
      \node at (0.0,-0.5) {\small $\G[\sigma]$};
    \end{scope}
    \draw[thin, dashed, decoration={
        markings,
        mark=at position 0.4 with {\arrow{>}}},
        postaction={decorate}
      ] (D) -- (d);
  \end{tikzpicture}
\]
In terms of transport of structure, two graphs $s\in \G[U]$ and
$t\in \G[V]$ are isomorphic if there is a bijection $\sigma:U\to V$
such that $\G[\sigma](s) = t$.

It is easy to check that for all bijections $\sigma: U\to V$ and
$\tau: V \to W$ we have
$\G[\sigma \circ \tau] = \G[\sigma]\circ \G[\tau]$. And, for the
identity $\Id_U:U\to U$, it holds that $\G[\Id_U] = \Id_{\G[U]}$.

A \emph{species} is a rule which

\begin{itemize}
\tightlist
\item
  produces, for each finite set $U$, a finite set $F[U]$;
\item
  produces, for each bijection $\sigma:U\to V$, a function
  $F[\sigma]:F[U]\to F[V]$.
\end{itemize}

These data satisfy the following conditions: $F[\tau\circ\sigma] = F[\tau] \circ F[\sigma]$ for
all bijections $\sigma: U\to V$ and $\tau: V\to W$, and
$F[\Id_U] = \Id_{F[U]}$ for the identity map $\Id_U:U\to U$. For
instance, $\G$ as defined above is a species.

While this definition of species is adequate, the natural setting for
species is category theory. A \emph{category} consists of
the following data:
\begin{itemize}
\tightlist
\item
  A collection of \emph{objects}: $U$, $V$, $W$, \dots
\item
  A collection of \emph{arrows}: $\sigma$, $\tau$, $\rho$, \dots,
  each having a \emph{domain} and a \emph{codomain}; writing
  $\sigma: U\to V$ we express that $U$ is the domain and $V$ is the
  codomain of $\sigma$.
\item
  For each pair of arrows $\sigma: U\to V$ and $\tau: V\to W$, a
  composite arrow $\tau\circ\sigma:U\to W$.
\item
  For each object $U$, an identity arrow $\Id_U:U\to U$.
\end{itemize}

These data satisfy the following axioms:

\begin{itemize}
\item
  Associativity: for $\sigma$ and $\tau$ as above and $\rho$ an
  arrow with domain $W$ we have
  $\rho\circ (\tau\circ\sigma) = (\rho\circ \tau)\circ\sigma$.
\item
  Identity laws: for $\sigma$ as above,
  $\sigma \circ \Id_U = \sigma = \Id_V \circ\hspace{1pt}\sigma$.
\end{itemize}

Let us now give three examples of categories, the latter two of which
are essential to the definition of a species. The category $\Set$ has
sets as objects and functions as arrows; the identity arrow $\Id_U$ is
the identity function on $U$. The category $\FinSet$ is defined
similarly but has only {\it finite} sets as objects.  The category
$\CoreFinSet$ has finite sets as objects and bijections as arrows.

Morphisms of categories are called a functors. Eilenberg and Mac
Lane stress their importance in {\it General theory of
natural equivalences}~\cite{Eilenberg1945}:

\begin{quote}
It should be observed first that the whole concept of a category is
essentially an auxiliary one; our basic concepts are essentially those
of a functor and of a natural transformation {[}\,\dots{]}
\end{quote}

The definition goes as follows: Let $\C$ and $\D$ be categories. A
\emph{functor} $F:\C\to \D$ is a mapping of objects to objects and
arrows to arrows such that

\begin{itemize}
\tightlist
\item
  If $\sigma:U\to V$ then $F[\sigma]:F[U]\to F[V]$;
\item
  $F[\tau\circ\sigma] = F[\tau]\circ F[\sigma]$;
\item
  $F[\Id_U]=\Id_{F[U]}$.
\end{itemize}

Armed with these notions from category theory a \emph{species} is simply
a functor $F:\CoreFinSet\to \FinSet$. This definition was first given
by Joyal in 1981~\cite{Joyal1981}. An element $s\in F[U]$ is called an
\emph{$F$-structure} on $U$. Note that if $F$ is a species and
$\sigma:U\to V$ is a bijection, then
$\Id_{F[U]} = F[\Id_U] = F[\sigma^{-1}\circ\sigma] = F[\sigma^{-1}]\circ F[\sigma]$
and, similarly, $\Id_{F[V]} = F[\sigma]\circ F[\sigma^{-1}]$. Thus,
the transport of structure is always a bijection and a species can also
be regarded as a functor $\CoreFinSet\to \CoreFinSet$. In particular,
if $U$ and $V$ are two finite sets of the same cardinality
$|U|=|V|$, and $\sigma:U\to V$ is a bijection witnessing this, then
$F[\sigma]:F[U]\to F[V]$ is a bijection, proving that
$|F[U]|=|F[V]|$. In other words, the number of $F$-structures on
$U$ only depends on the number of elements of $U$.

For ease of notation, we write $F[n]$ for $F[[n]]$, the set of
$F$-structures on $[n]=\{1,2,\dots,n\}$. The \emph{generating
series} of $F$ is the formal power series
\[F(x) = \sum_{n\geq 0} |F[n]|\frac{x^n}{n!}.
\]
Two $F$-structures $s\in F[U]$ and $t\in F[V]$ are
\emph{isomorphic}, and we write $s\sim t$, if there is a bijection
$\sigma:U\to V$ such that $F[\sigma](s) = t$. Let $F[U]/{\sim}$
denote the set of equivalence classes under this relation. An unlabeled
$F$-structure is then simply a member of $F[U]/{\sim}$. Two
isomorphic $F$-structures are also said to have the same
\emph{isomorphism type} and the \emph{type generating series} of $F$
is
\[\typ{F}(x) = \sum_{n\geq 0} |F[n]/{\sim}|x^n.
\]

We have only seen one concrete species so far, that of simple graphs. Another
example is the species $L$ of linear orders. For a finite set $U$
with $n$ elements, let $L[U]$ consist of all bijections
$f:[n]\to U$. For instance, if we represent $f$ by its images
$f(1)f(2)\dots f(n)$ then $L[3]=\{123,132,213,231,312,321\}$.
Transport of structure is defined by $L[\sigma](f) = \sigma \circ f$,
or in terms of images
$L[\sigma](f(1)\dots f(n)) = \sigma(f(1))\dots \sigma(f(n))$. Clearly,
there are $n!$ linear orders on $[n]$. Thus $|L[n]|=n!$ and
$L(x) = 1/(1-x)$. Any two linear orders $f,g\in L[U]$ are isomorphic.
Indeed, $\sigma = g\circ f$ defines a bijection from $U$ to $U$
and $g = L[\sigma](f)$. Thus $|L[n]/{\sim}|=1$ and
$\typ{L}(x) = 1/(1-x)$.

Let $E$ be the species of sets. It is defined by $E[U]=\{U\}$ and,
for bijections $\sigma:U\to V$, $E[\sigma](U)=V$. That is, for any
finite set $U$ there is exactly one $E$-structure, the set $U$
itself. Thus $E(x)=e^x$ and $\typ{E}(x)=1/(1-x)$.

We can build new species from existing species using certain operations.
The simplest operation is addition, which is just disjoint union. An
$(F+G)$-structure is thus either an $F$-structure or a
$G$-structure:
\[
  \begin{tikzpicture}[semithick, scale=0.7]
    \draw[rounded corners] (0,0) rectangle (3,2);
    \foreach \x/\y in
      {0.35/1.64,0.26/0.3,0.9/1.7,1.16/0.3,0.4/1,2.4/1.2,2.0/1.7,2.3/0.4,2.6/0.45} {
      \filldraw[black!40!green] (\x,\y) circle (1.4pt);
    };
    \node at (1.5,1) {$F+G$};
    \node at (3.5,1) {$=$};
    \foreach \x/\y in
      {0.35/1.64,0.26/0.3,0.9/1.7,1.16/0.3,0.4/1,2.4/1.2,2.0/1.7,2.3/0.4,2.6/0.45} {
      \filldraw[black!40!green] (4+\x,\y) circle (1.4pt);
    };
    \node at (5.5,1) {$F$};
    \draw[rounded corners] (4,0) rectangle +(3,2);
    \node at (7.5,1) {or};
    \foreach \x/\y in
      {0.35/1.64,0.26/0.3,0.9/1.7,1.16/0.3,0.4/1,2.4/1.2,2.0/1.7,2.3/0.4,2.6/0.45} {
      \filldraw[black!40!green] (8+\x,\y) circle (1.4pt);
    };
    \node at (9.5,1) {$G$};
    \draw[rounded corners] (8,0) rectangle +(3,2);
  \end{tikzpicture}
\]
That is, for any finite set $U$ we define
$(F+G)[U] = F[U] \sqcup G[U]$. For any bijection $\sigma:U\to V$,
transport of structure is defined by $(F+G)[\sigma](s) = F[\sigma](s)$
if $s\in F[U]$ and $(F+G)[\sigma](s) = G[\sigma](s)$ if
$s\in G[U]$.

We can also form the product of two species. An $(F\cdot G)$-structure
on a set $U$ is a pair $(s,t)$, where $s$ is an $F$-structure on
a subset $U_1$ of $U$ and $t$ is a $G$-structure on the
remaining elements $U_2=U\setminus U_1$:
\[
  \begin{tikzpicture}[semithick, scale=0.7]
    \draw[rounded corners] (0,0) rectangle (3,2);
    \foreach \x/\y in
      {0.35/1.64,0.26/0.3,0.9/1.7,1.16/0.3,0.4/1,2.4/1.2,2.0/1.7,2.3/0.4,2.6/0.45} {
      \filldraw[black!40!green] (\x,\y) circle (1.4pt);
    };
    \node at (1.5,1) {$F\cdot G$};
    \node at (3.5,1) {$=$};
    \foreach \x/\y in
      {0.35/1.64,0.26/0.3,0.9/1.7,1.16/0.3,0.4/1,2.4/1.2,2.0/1.7,2.3/0.4,2.6/0.45} {
      \filldraw[black!40!green] (4+\x,\y) circle (1.4pt);
    };
    \draw [zigzag] (5.45,0) -- (5.7,2);
    \node at (4.9,1) {$F$};
    \node at (6.04,1) {$G$};
    \draw[rounded corners] (4,0) rectangle +(3,2);
  \end{tikzpicture}
\]
Formally, $(F\cdot G)[U] = \bigsqcup (F[U_1] \times G[U_2])$ in
which the union is over all pairs $(U_1,U_2)$ such that
$U = U_1\cup U_2$ and $U_1\cap U_2 = \emptyset$. The transport of
structure is defined by
$(F\cdot G)[\sigma](s,t) = \bigl(F[\sigma_1](s), G[\sigma_2](t)\bigr)$,
where $\sigma_1 = \sigma|_{U_1}$ is the restriction of $\sigma$ to
$U_1$ and $\sigma_2 = \sigma|_{U_2}$ is the restriction of
$\sigma$ to $U_2$.

It is clear that if $F$ and $G$ are species, then
$(F+G)(x) = F(x)+G(x)$. It also holds that $(FG)(x) = F(x)G(x)$. To
see this, recall that, for any species $F$ and any finite set $U$,
the number $|F[U]|$ of $F$-structures on $U$ only depends on
$|U|$. Thus, the coefficient in front of $x^n/n!$ in $(FG)(x)$ is
\[
  \big|(FG)[n]\big|
    = \biggl|\bigsqcup_{S\subseteq [n]} F\big[S\big] \times G\big[[n]\setminus S\big]\,\biggr|
    = \sum_{k=0}^n\binom{n}{k}\big|F[k]\big|\,\big|G[n-k]\big|,
\]
which is also the coefficient in front of $x^n/n!$ in $F(x)G(x)$.
For the type generating series one can similarly show that
$(\typ{F+G})(x)=\typ{F}(x)+\typ{G}(x)$ and
$(\typ{FG})(x)=\typ{F}(x)\typ{G}(x)$.

As a simple example, multiplying the set species with itself we get
$E^2[U] = \bigl\{(A, U\setminus A): A\subseteq U \bigr\}$. This is
called the subset (or power set) species, and is denoted $\Pow$. Note
that $\Pow(x) = E(x)^2 = e^{2x} = \sum_{n\geq 0}2^n x^n/n!$,
reflecting the elementary fact that there are $2^n$ subsets of an
$n$-element set. Also,
$\typ{\Pow}(x) = 1/(1-x)^2 = \sum_{n\geq 0}(n+1)x^n$, reflecting that
the isomorphism type of a set is determined by its cardinality
and that an $n$-element set has subsets of $n+1$ different cardinalities.

Another, perhaps more interesting, example is that of derangements:
Define the species $\Sym$ of permutations as follows. For any finite
set $U$, let $\Sym[U]$ be the set of bijections $f: U\to U$; and,
for any bijection $\sigma: U\to V$, define the transport of structure
by $\Sym[\sigma](f)=\sigma\circ f\circ \sigma^{-1}$. A
\emph{derangement} is a permutation without fixed points. Let $\Der$
denote the species of derangements, so that
$\Der[U] = \{f\in \Sym[U]: \text{$f(x)\neq x$ for all $x\in U$}\}$
with transport of structure defined the same way as for the species
$\Sym$. Note that any permutation on $U$ can be seen as a set of
fixed points on a subset $U_1$ of $U$ together with a derangement of
the remaining elements $U_2=U\setminus U_1$. This is most apparent
when writing a permutation in disjoint disjoint cycle form. For
instance, the permutation $(1)(2\, 7\, 3)(4)(5)(6\, 9)(8)$ can be
identified with a the pair consisting of the fixed point set
$\{1,4,5,8\}$ together with the derangement $(2\, 7\, 3)(6\, 9)$.
Thus $\Sym = E\cdot \Der$ and $\Sym(x) = E(x)\Der(x)$. It
immediately follows that \begin{align*}
    \Der(x)
    = \Sym(x)\cdot E(x)^{-1}
    &= \frac{1}{1-x}\cdot e^{-x} \\
    &=
      \sum_{n\geq 0}n!\left(1 - \frac{1}{1!} + \frac{1}{2!} - \dots + \frac{(-1)^n}{n!}
      \right)\frac{x^n}{n!}.
\end{align*} In particular, $|Der[n]|/n! \to e^{-1}$ as
$n\to \infty$, which the reader may recognize as the answer to
Montmort's (1713) hat-check problem~\cite{Montmort1713}.

Define the singleton species $X$ by $X[U] = \{U\}$ if $|U|=1$ and
$X[U] = \emptyset$ otherwise. Similarly, define the species $1$,
characteristic of the empty set, by $1[U] = \{U\}$ if $U=\emptyset$
and $1[U] = \emptyset$ otherwise. The combinatorial equality
$L=1+XL$ informally states that a linear order is either empty or it
consists of a first element together with the remaining elements. A
computer scientist would recognize this as the same recursive pattern as
in the definition of a singly linked list of nodes: it is either the
null reference or it is a pair consisting of a data element and a
reference to the next node in the list. To make formal sense of this
combinatorial equality we first need to define what it means
for two species to be equal.

Two species $F$ and $G$ are \emph{(combinatorially) equal} if there
is a family of bijections $\alpha_U:F[U]\to G[U]$ such that for any
bijection $\sigma:U\to V$ the diagram
\[
  \begin{tikzpicture}[semithick]
    \matrix (m) [matrix of math nodes, row sep=4em, column sep=6em]
    { F[U] & G[U] \\
      F[V] & G[V] \\ };
    \path[->,font=\scriptsize]
      (m-1-1) edge node[above] {$\alpha_U$} (m-1-2)
              edge node[left]  {$F[\sigma]$} (m-2-1)
      (m-2-1) edge node[above] {$\alpha_V$} (m-2-2)
      (m-1-2) edge node[right] {$G[\sigma]$} (m-2-2);
  \end{tikzpicture}
\]
commutes. That is, for any $F$-structure $s$ on $U$ we have
$G[\sigma](\alpha_U(s)) = \alpha_V(F[\sigma](s))$. In category theory
this is called natural isomorphism of functors.

To prove that $L=1+XL$ we define the family of bijections
$\alpha_U:L[U]\to (1+XL)[U]$ by $\alpha_\emptyset(\eps) = \emptyset$
and $\alpha_U(a_1 a_2 \cdots a_n) = (a_1,\, a_2\cdots a_n)$ for
$U\neq \emptyset$. In the case $U=\emptyset$ the diagram above
trivially commutes and for any nonempty finite set $U$ we have
\[
  \begin{tikzpicture}[semithick]
    \matrix (m) [matrix of math nodes, row sep=3em, column sep=4em, minimum height=2em]
    { \; a_1 a_2 \cdots a_n\, & \,(a_1,\; a_2\cdots a_n)\; \\
      \sigma(a_1)\sigma(a_2)\cdots \sigma(a_n)
      & (\sigma(a_1),\; \sigma(a_2)\cdots \sigma(a_n)) \\ };
    \path[->,font=\scriptsize]
      (m-1-1) edge[|->] node[above] {$\alpha_U$} (m-1-2)
              edge[|->] node[left=3pt] {$L[\sigma]$} (m-2-1)
      (m-2-1) edge[|->] node[above] {$\alpha_V$} (m-2-2)
      (m-1-2) edge[|->] node[right=3pt] {$(1+XL)[\sigma]$} (m-2-2);
  \end{tikzpicture}
\]
which concludes the proof of $L = 1+XL$. From this it immediately
follows that $L(x) = 1 + xL(x)$ and we rediscover the identity
$L(x) = 1/(1-x)$. Similarly, $\typ{L}(x) = 1 + x\typ{L}(x)$ and we
also rediscover that $\typ{L}(x) = 1/(1-x)$.

Next we shall define substitution of species. As a motivating example,
consider the species $\Sym$ of permutations. For instance,
$(1\, 7)(2\, 8\, 3\, 4)(5\, 9)(6)$ is an element of $\Sym[9]$
written in disjoint cycle form. Here we have listed the cycles in
increasing order with respect to the smallest element in each cycle, but
this is just a convention; since the cycles are disjoint they commute
and we may list them in any order. In other words, any permutation can
be represented by a \emph{set} of cycles. Using species this can be
expressed as $\Sym = E\circ \Cyc = E(\Cyc)$, in which a
$\Cyc$-structure is a permutation with a single
cycle.

Another example of substitution of species is that of set partitions;
this species is denoted $\Par$. Define $E_+$ as the species of
nonempty sets, so that $E=1+E_+$. Then $\Par = E(E_+)$, which we can
read as stating that a partition is a set of nonempty sets. Similarly,
the species of \emph{ballots}, also called ordered set partitions, is
defined by $\Bal = L(E_+)$.

In general, we can think of an $F(G)$-structure as a generalized
partition in which each block of the partition carries a
$G$-structure, and the blocks are structured by $F$:
\[
  \begin{tikzpicture}[semithick, scale=0.8]
    \draw[rounded corners] (0,0) rectangle (3,2);
    \foreach \x/\y in
      {0.35/1.64,0.26/0.3,0.9/1.7,1.16/0.3,0.4/1,2.4/1.2,2.0/1.7,2.3/0.4,2.6/0.45} {
      \filldraw[black!40!green] (\x,\y) circle (1.4pt);
    };
    \node at (1.5,1) {$F\circ G$};
    \node at (3.5,1) {$=$};
    \foreach \x/\y in
      {0.35/1.64,0.26/0.3,0.9/1.7,1.16/0.3,0.4/1,2.4/1.2,2.0/1.7,2.3/0.4,2.6/0.45} {
      \filldraw[black!40!green] (4+\x,\y) circle (1.4pt);
    };
    \draw [zigzag] (5.45,0) -- (5.85,2);
    \draw [zigzag] (4,0.44) -- (5.7,1.5);
    \draw [zigzag] (5.65,1.07) -- (7,0.8);
    \node at (4.65,1.33) {\small $G$};
    \node at (4.78,0.55) {\small $G$};
    \node at (6.60,1.6) {\small $G$};
    \node at (6.00,0.4) {\small $G$};
    \node at (5.50,1.14) {\Large $F$};
    \draw[rounded corners] (4,0) rectangle +(3,2);
  \end{tikzpicture}
\]
Formally, we define the \emph{substitution} of $G$ in $F$, also
called \emph{partitional composition}, as follows. For species $F$ and
$G$, with $G[\emptyset]=\emptyset$,
\[(F\circ G)[U] \,=
  \bigsqcup_{\substack{\beta = \{B_1,B_2,\dots,B_k\} \\ \text{partition of }U}}
  F[\beta]\times G[B_1]\times G[B_2] \times\cdots\times G[B_k].
\]
Transport of structure is defined in a natural way, but there are
quite a few details to keep track of. To a first approximation, an
$(F\circ G)$-structure on $U$ is a tuple $(s,t_1,\dots,t_k)$,
where $s\in F[\beta]$ and $t_i\in G[B_i]$. To guarantee that the
union is disjoint we will, however, also ``tag'' these tuples with their
associated set partition. Hence a typical $(F\circ G)$-structure is of
the form $(\beta,s,t_1,\dots,t_k)$.
To describe the image of this tuple under the transport of structure we
first need some auxiliary definitions.  Assume that $\sigma:U\to V$ is a
bijection.  Let $\bar\beta=\{{\bar B}_1,\dots,{\bar B}_k\}$ be the set
partition of $V$ obtained from $\beta$ by transport of structure. That is, with
$\bar\sigma = \Par[\sigma]$ we have $\bar\beta
=\bar\sigma(\beta)$. Further, let
$\bar s = F[\bar\sigma](s)\in F[\bar\beta]$ and
${\bar t}_i = G[\sigma_i](t_i)\in G[{\bar B}_i]$, where
$\sigma_i = \sigma|_{B_i}$ is the restriction of $\sigma$ to the block
$B_i$. Then
\[
(\bar\beta,\bar s,\bar{t}_1,\dots,\bar{t}_k) = (F\circ G)[\sigma](\beta,s,t_1,\dots,t_k).
\]

Concerning generating series, one can show that $F(G)(x)=F(G(x))$. As
an example, recall that $\Sym = E\circ \Cyc$ and hence
$\Sym(x) = \exp(\Cyc(x))$. Thus
\[\Cyc(x) = \log(\Sym(x)) = \log\,(1-x)^{-1}
= \sum_{n\geq 1} \frac{x^n}{n}
= \sum_{n\geq 1} (n-1)!\frac{x^n}{n!},
\]
which is unsurprising since there clearly are $(n-1)!$ ways to make a
cycle using $n$ distinct elements.

For type generating series the situation is more complicated. The
series $\typ{F(G)}(x)$ is, as a rule, different from
$\typ{F}(\typ{G}(x))$. To state the true relation between these
entities a more general generating series is needed. The \emph{cycle
index series} for a species $F$ is defined by \begin{equation*}
  Z_{F}(x_1,x_2,x_3,\dots) =
    \sum_{n\geq 0}\,\frac{1}{n!}
    \sum_{\sigma\in\Sym[n]}|\Fix F[\sigma]|\,
      x_1^{c_1(\sigma)}x_2^{c_2(\sigma)}x_3^{c_3(\sigma)}\cdots
\end{equation*} where
$\Fix F[\sigma] = \{ s \in F[n]: F[\sigma](s) = s \}$ is the set of
$F$-structures on $[n]$ fixed by $F[\sigma]$, and $c_i(\sigma)$
denotes the number of cycles of length $i$ in $\sigma$. To
illustrate this definition we give three examples:

\emph{Example 1}. Consider the species $X$ of singletons. The fixed
point set $\Fix X[\sigma]$ is empty unless $n=1$ and there is only
one permutation on $[1]=\{1\}$, the identity, so the cycle index
series consists of a single term:
\begin{align*}
  Z_X(x_1, x_2, \dots)
    &= |\Fix X[\Id_1]|\,x_1^{c_1(\Id_1)} x_2^{c_2(\Id_1)}\dots = x_1.
\end{align*}

\emph{Example 2}. For any species $F$ we denote by $F_n$ the
restriction of $F$ to sets of cardinality $n$. In other words,
$F_n[U] = F[U]$ if $|U|=n$ and $F_n[U] = \emptyset$ otherwise.
With this notation the singleton species $X$ is identical to the
species $E_1$. Let us now calculate the cycle index series for
$E_2$, the species of sets of cardinality 2:
\begin{align*}
  Z_{E_2}(x_1, x_2, \dots)
    &= \frac{1}{2}\left(|\Fix E_2[\Id_2]|\,x_1^{c_1(\Id_2)} x_2^{c_2(\Id_2)} +
       |\Fix E_2[(1\,2)]|\,x_1^{c_1(1\,2)} x_2^{c_2(1\,2)}\right) \\
    &= \frac{1}{2}\left(x_1^2 + x_2\right).
\end{align*}

\emph{Example 3}. For the species $L$ of linear orders,
$\Fix L[\sigma]$ is empty unless $\sigma$ is the identity map, in
which case $\Fix L[\sigma] = L[n]$. Thus
\begin{align*}
  Z_L(x_1, x_2, \dots)
    &= \sum_{n \geq 0}\frac{1}{n!}|\Fix L[\Id_n]|\, x_1^{c_1(\Id_n)} x_2^{c_2(\Id_n)}\dots \\
    &= \sum_{n \geq 0}\frac{1}{n!}|L[n]| x_1^{n} = \frac{1}{1-x_1}.
\end{align*}

It should be noted that the cycle index series encompasses both the
generating series and the type generating series of a species. To be
more precise, for any species $F$,
\[
F(x) = Z_F(x,0,0,\dots) \;\text{ and }\;
\typ{F}(x) = Z_F(x,x^2,x^3,\dots).
\]
Let us prove this. For the generating series we have
\begin{align*}
  Z_F(x,0,0,\dots)
      &= \sum_{n \geq 0} \frac{1}{n!}\Biggl(\sum_{\,\sigma \in \Sym[n]}
        |\Fix F[\sigma]| x^{c_1(\sigma)} 0^{c_2(\sigma)} 0^{c_3(\sigma)}\!\dots\Biggr)\\
      &= \sum_{n \geq 0}\frac{1}{n!} |\Fix F[\Id_n]| x^n
      = \sum_{n \geq 0} |F[n]| \frac{x^n}{n!}
      = F(x).
\end{align*}
For the type generating series we have
\begin{align*}
  Z_F(x,x^2,x^3,\dots)
      &= \sum_{n \geq 0} \frac{1}{n!} \sum_{\sigma \in \mathcal{S}_n} |\Fix F[\sigma]| x^{c_1(\sigma)+2c_2(\sigma)+3c_3(\sigma)+\cdots} \\
      &= \sum_{n \geq 0} \frac{1}{n!} \sum_{\sigma \in \mathcal{S}_n} |\Fix F[\sigma]| x^n
      = \sum_{n \geq 0} |F[n]/\mathord{\sim}| x^n = \typ{F}(x)
\end{align*}
by Burnside's lemma.

Further, for species $F$ and $G$ we have $Z_{F+G} = Z_F+ Z_G$,
$Z_{F\cdot G} = Z_F \cdot Z_G$, and
\[
  Z_{F\circ G}(x_1, x_2, \dots) = Z_F(Z_G(x_1,x_2,\dots), Z_G(x_2,x_4,\dots), \dots).
\]
In particular, $(F\circ G)(x) = F(G(x))$ while
$\typ{F\circ G}(x) = Z_F(\typ{G}(x), \typ{G}(x^2), \typ{G}(x^3), \dots)$.
As an example, consider the species $F=E_2\circ (X+X^2)$ and
isomorphism types of $F$-structures. For each of the sizes 2, 3, and 4
there is exactly one type: using pictures, those are
$\{\bullet,\bullet\}$, $\{\bullet,\bullet\bullet\}$, and
$\{\bullet\bullet,\bullet\bullet\}$. There is no type of any other
size and thus the type generating series is $\typ{F}(x)=x^2+x^3+x^4$.
Note that the type generating series for $E_2$ and $X+X^2$ are
$x^2$ and $x+x^2$, and naively composing them we get
$x^2+2x^3+x^4$, which is different from $\typ{F}(x)$. The correct
procedure to calculate $\typ{F}(x)$ is
\begin{align*}
  \typ{F}(x)
  &= Z_{E_2}\big(\typ{(X+X^2)}(x), \typ{(X+X^2)}(x^2), \dots\big) \\
  &= Z_{E_2}\big(x+x^2, x^2+x^4\big)
  = \frac{1}{2}\Big(\big(x+x^2\big)^2 + x^2+x^4\Big) = x^2+x^3+x^4.
\end{align*}

Changing the topic, but continuing with the species
$F=E_2\circ (X+X^2)$, say that we want to refine the count of
$F$-structures by keeping track of how many ordered pairs, stemming
from the $X^2$ term, such a structure contains. We can accomplish this
by putting a ``weight'', say $y$, on each $X^2$-structure. The
type generating series for $F$ would in this case be
$x^2 + yx^3 + y^2x^4$. Two more
interesting examples of this kind of refinement would be
to count permutations while
keeping track of the number of cycles, or to count trees while keeping
track of the number of leaves. There is a variant of combinatorial
species, called weighted species, that allows us to keep track of
various parameters. Before giving its definition we'll need to introduce
weighted sets and their morphisms.

Let $R$ be a ring. An \emph{$R$-weighted set} is a pair $(A,w)$,
where $A$ is a finite set and $w:A\to R$ is a function. We think of
$w$ as assigning a weight to each element of $A$. A \emph{morphism
of $R$-weighted sets} $(A,w)$ and $(B,v)$ is a
\emph{weight-preserving} function $\vphi:A\to B$; that is,
$w = v \circ \vphi$, or, equivalently, the diagram
\[
  \begin{tikzpicture}[semithick]
    \matrix (m) [matrix of math nodes, row sep=2.5em, column sep=3em]
    { A &   & B \\
        & R \\ };
    \path[->,font=\scriptsize]
      (m-1-1) edge node[auto] {$\vphi$} (m-1-3)
      (m-1-1) edge node[below left] {$w$} (m-2-2)
      (m-1-3) edge node[below right] {$v$} (m-2-2);
  \end{tikzpicture}
\]
commutes. If, in addition, $\psi$ is a morphism of $R$-weighted
sets from $(B,v)$ to $(C,u)$ then it's easy to check that the
composition $\psi\circ\vphi$ is a morphism from $(A,w)$ to
$(C,u)$. The identity map is of course also weight-preserving and so
we have all the ingredients of a category. More generally, given
categories $\A$ and $\D$, an object $D\in\D$ and a functor
$F:\A\to \D$, the \emph{slice category} $(F/D)$ is defined as
follows: Objects are pairs $(A,w)$ with $A\in\A$ and $w:F[A]\to D$
in $\D$. Arrows $(A,w)\to (A',w')$ are arrows $\vphi:A\to A'$
making the following diagram commute:
\[
  \begin{tikzpicture}[semithick]
    \matrix (m) [matrix of math nodes, row sep=2.5em, column sep=3em]
    { F[A] &   & F[A'] \\
           & D &       \\ };
    \path[->,font=\scriptsize]
      (m-1-1) edge node[above] {$F[\vphi]$} (m-1-3)
      (m-1-1) edge node[below left] {$w$} (m-2-2)
      (m-1-3) edge node[below right] {$w'$} (m-2-2);
  \end{tikzpicture}
\]
Recall that $\Set$ is the category of sets and functions and that
$\FinSet$ is the category of finite sets and functions.  By slight abuse
of notation, let the inclusion functor from $\FinSet$ to $\Set$ taking
objects and arrows to themselves also be denoted by $\FinSet$.  In terms
of the definition of a slice category, let $F=\FinSet$, $\A = \FinSet$,
and $\D = \Set$. Assuming that we for the moment identify the ring $R$ with its
underlying set, forgetting about the two operations on $R$, we
also let $D=R$. Then $(\FinSet/R)$ is the (slice) category of
$R$-weighted sets with morphisms of $R$-weighted sets as arrows, and we
define an \emph{$R$-weighted species} as a functor
$F:\CoreFinSet\to(\FinSet/R)$.

As an example, let us define a $\ZZ[y]$-weighted version of the
species $\Sym$ of permutations. Define the weight of $f\in \Sym[U]$
as $w_U(f)=y^{\cyc(f)}$, where $\cyc(f)$ denotes the number of
cycles in $f$, and let $\Sym_w$ denote the resulting weighted
species. To be more specific,

\begin{itemize}
\tightlist
\item
  for any finite set $U$, let $\Sym_w[U]$ be the $\ZZ[y]$-weighted
  set $\big(\Sym[U], w_U\big)$;
\item
  for any bijection $\sigma:U\to V$, let
  $\Sym_w[\sigma]:\Sym_w[U]\to \Sym_w[V]$ be the morphism of
  $\ZZ[y]$-weighted sets defined by
  $\Sym[\sigma]:\Sym[U]\to \Sym[V]$.
\end{itemize}

Recall that the transport of structure $\Sym[\sigma]$ is conjugation
by $\sigma$; that is,
$\Sym[\sigma](f)=\sigma\circ f\circ \sigma^{-1}$. It is well known
that two permutations are conjugate if and only if they have the same
cycle structure. In particular, conjugate permutations have the same
number of cycles. Thus, $\Sym[\sigma]$ is weight-preserving, or,
equivalently, the diagram
\[
  \begin{tikzpicture}[semithick]
    \matrix (m) [matrix of math nodes, row sep=2.5em, column sep=3em]
    { \Sym[U] &        & \Sym[V] \\
              & \ZZ[y] &         \\
    };
    \path[->,font=\scriptsize]
      (m-1-1) edge node[auto] {$\Sym[\sigma]$} (m-1-3)
      (m-1-1) edge node[below left]  {$w_U$} (m-2-2)
      (m-1-3) edge node[below right] {$w_V$} (m-2-2);
  \end{tikzpicture}
  \vspace{-1.5ex}
\]
commutes.

For any weighted set $(A,w)$ let $|A|_w = \sum_{a\in A}w(a)$; it's
called the \emph{total weight} of $A$. Note that if $w$ is the
trivial weight function, assigning $1$ to each element, then
$|A|_w=|A|$. With this definition in hand, the power series associated
with a weighted species $F_w$ are defined as expected:
\begin{align*}
  Z_{F_w}(x_1,x_2,x_3,\dots) &=
    \sum_{n\geq 0}\,\frac{1}{n!}
    \sum_{\sigma\in\Sym[n]}|\Fix F[\sigma]|_w\,
      x_1^{c_1(\sigma)}x_2^{c_2(\sigma)}x_3^{c_3(\sigma)}\cdots, \\
  F_w(x) = Z_{F_w}(x,0,0,\dots) &= \sum_{n\geq 0} |F[n]|_w\frac{x^n}{n!},\\
  \typ{F}_w(x) = Z_{F_w}(x,x^2,x^3,\dots) &= \sum_{n\geq 0} |F[n]/{\sim}|_w x^n,
\end{align*}
where $|F_w[U]/{\sim}|_w$ is the total weight of any
member of $F_w[U]/{\sim}$. Since $F_w[\sigma]$ is weight-preserving
the choice of representative is immaterial.

It is straightforward to extend addition, multiplication, and
composition to weighted species. All we need to do is define disjoint
union and (Cartesian) product of weighted sets: If $(A,w)$ and
$(B,v)$ are weighted sets, then
$(A,w)\sqcup (B,v) = (A\sqcup B, w\sqcup v)$ and
$(A,w)\times (B,v) = (A\times B, w\times v)$, where $(w\sqcup v)(x)$
is $w(x)$ or $v(x)$ depending on whether $x\in A$ or $x\in B$,
and $(w\times v)(x,y) = w(x)v(y)$. With these definitions we have
$|A\sqcup B|_{w\sqcup v} = |A|_w + |B|_v$ and
$|A\times B|_{w\times v} = |A|_w|B|_v$.

Returning to the $\Sym_w$ example, we find that $\Sym_w=E(y\Cyc)$,
where $y\Cyc$ is suggestive notation for the weighted species of
cycles in which each cycle is given the weight $y$. Thus
$\Sym_w(x)=\exp(y\Cyc(x)) =\exp(y\log (1-x)^{-1}) = (1-x)^{-y}$ and
the coefficient in front of $y^kx^n/n!$ in $\Sym_w(x)$ is
an unsigned Stirling number of the first kind. That is,
$$
  \Sym_w(x) = \sum_{n\geq 0}\left(\sum_{k=1}^n c(n,k)y^k\right)\frac{x^n}{n!},
$$
where $c(n,k)$ is the number of permutations of $[n]$ with $k$ cycles.

As a more complex example we shall now calculate the cycle index series of
the set species. We have $|\Fix E[\sigma]|=1$ for every
$\sigma\in\Sym[n]$ and so
\[
    Z_{E}(x_1,x_2,x_3,\dots) =
      \sum_{n\geq 0}\,\frac{1}{n!}
      \sum_{\sigma\in\Sym[n]}x_1^{c_1(\sigma)}x_2^{c_2(\sigma)}x_3^{c_3(\sigma)}\cdots
\]
Or, rewriting this slightly,
\[
    Z_{E}(x_1,x_2,x_3,\dots) =
    \left[\,
      \sum_{n\geq 0}\,\left(
      \sum_{\sigma\in\Sym[n]}x_1^{c_1(\sigma)}x_2^{c_2(\sigma)}x_3^{c_3(\sigma)}\cdots
      \right)\!\frac{x^n}{n!}\,
    \right]_{x=1}.
\]
Inside the outermost bracket is an exponential generating series for
permutations in which each cycle of length $i$ carries the weight
$x_i$. Next, we recast this example in terms of weighted species. Let
$\Sym_w$ be the $\ZZ[x_1,x_2,\dots]$-weighted species of
permutations in which
$w(\sigma) = x_1^{c_1(\sigma)}x_2^{c_2(\sigma)}\cdots$ for each
permutation $\sigma$. Then $Z_{E}(x_1,x_2,\dots) = \Sym_w(1)$. On
the other hand, we have $\Sym_w = E(x_1\Cyc_1 + x_2\Cyc_2 +\cdots)$ in
which $x_i\Cyc_i$ is the species of cycles of length $i$, each with
weight $x_i$. Thus
$\Sym_w(x) = \exp\bigl(\frac{x_1x}{1}+\frac{x_2x^2}{2}+\frac{x_3x^3}{3} + \cdots\bigr)$
and setting $x=1$ we arrive at
\[
    Z_{E}(x_1,x_2,x_3,\dots)
    = \exp\!\left( \frac{x_1}{1}+\frac{x_2}{2}+\frac{x_3}{3} + \cdots \right).
\]
Let us use this expression to investigate the type generating series
of the set partition species $\Par$. The isomorphism type of a set
partition is determined by the sizes of its blocks; those sizes form an
integer partition and hence $\uPar(x) = \prod_{k\geq 1}(1-x^k)^{-1}$.
On the other hand, $\Par=E(E_+)$ which in terms of type generating
series means
\[
  \uPar(x) = Z_E(\uNonEmptyE(x),\uNonEmptyE(x^2),\uNonEmptyE(x^3),\dots).
\]
Clearly, $\uNonEmptyE(x) = x/(1-x)$ and we arrive at the intriguing
identity
\[
\prod_{k\geq 1}\frac{1}{1-x^k}
  = \exp\sum_{k\geq 1}\,\frac{1}{k}\,\frac{x^k}{1-x^k}.
\]

Let $\Par_w$ be the species of set partitions weighted by the number
of blocks: the weight function $w_U:\Par[U]\to \ZZ[y]$ is defined by
$w_U(\beta)=y^k$, where $\beta=\{B_1,\dots,B_k\}$ is a partition of
$U$ with $k$ blocks. Then $\Par_w = E(yE_+)$, in which $yE_+$
denotes the species of nonempty sets each with weight $y$. Thus
$\Par_w(x) = exp\big(y(e^x-1)\big)$, and the coefficient in front of
$y^kx^n/n!$ in $\Par_w(x)$ is a Stirling number of the second kind:
$$
  \Par_w(x) = \sum_{n\geq 0}\left(\sum_{k=1}^n S(n,k)y^k\right)\frac{x^n}{n!},
$$
where $S(n,k)$ is the number of set partitions of $[n]$ with $k$ blocks.
Transitioning to type generating series we also find that
\begin{equation}\label{eq:integer-partitions-with-k-blocks}
\prod_{k\geq 1}\frac{1}{1-yx^k}
  = \exp\sum_{k\geq 1}\,\frac{1}{k}\,\frac{y^k x^k}{1-x^k},
\end{equation}
which is a refined generating series for integer partitions in which
$y$ keeps track of the number of parts.

A two-sort species is a functor
$F:\CoreFinSet\times\CoreFinSet\to\FinSet$ in which
$\CoreFinSet\times\CoreFinSet$ denotes a product category. In general,
if $\A$ and $\B$ are two categories, then $\A\times\B$ has pairs
$(A,B)$ as objects, where $A$ is an object of $\A$ and $B$ is an
object of $\B$. Similarly, the arrows are pairs $(f,g)$, where $f$
and $g$ are arrows of $\A$ and $\B$, respectively. Composition of
arrows is componentwise and $\Id_{(A,B)} = (\Id_A, \Id_B)$. We shall
write $F(X,Y)$ to indicate that $F$ is a two-sort species and that
the two sorts are called $X$ and $Y$. For natural numbers $n$ and
$k$, let $F[n,k] = F[[n],[k]]$. We define the \emph{generating
series} of $F$ by
\[
  F(x,y) = \sum_{n,k\geq 0} \big|F[n,k]\big| \,\frac{x^n}{n!}\frac{y^k}{k!}.
\]
Note that any unisort species $F$ can be considered a two-sort
species by specifying of what sort it is: saying that $F$ is a species
of sort $X$ means that, as a two-sort species, $F[U,V] = F[U]$ if
$V$ is empty, and $F[U,V]=\emptyset$ otherwise. To be a species of
sort $Y$ is interpreted similarly.

Addition and multiplication of species is straightforwardly extended to
the two-sort context. As an example, consider the species
$F(X,Y)=(1+X)Y$, where $X$ is the species of singletons of sort
$X$ and $Y$ is the species of singletons of sort $Y$. If
$|U|\leq 1$ and $|V| = 1$ then there is exactly one $F$-structure,
namely the pair $(U,V)$; otherwise, there are no $F$-structures. For
instance, $F[\emptyset,\{1\}] = \{(\emptyset, \{1\})\}$ and
$F[\{1,2\},V] = \emptyset$. Similarly,
\[
\big(E(X)Y\big)[U,V] =
\begin{cases}
  \big\{(U,V)\big\}  & \text{if $|V|=1$}, \\
  \emptyset & \text{otherwise}.
\end{cases}
\]

It is less clear how to extend the substitution operation to the
two-sort context and so we detail it here. Let $\Par[U,V]$ denote the
set of partitions of the disjoint union $U\sqcup V$. For instance, if
$U=\{1,2\}$ and $V = \{a,b,c\}$, then
$\pi = \big\{\{1\},\{2,a,c\},\{b\}\big\}$ is a member of
$\Par[U,V]$. In the definition of $F(G,H)$ below, we view a block $B$
of such a partition as a pair of sets $(S,T)$, where
$S=B\cap U \subseteq U$ and $T=B\cap V\subseteq V$. Thus, the blocks
of the example partition, $\pi$, will be viewed as the pairs
$(\{1\}, \emptyset)$, $(\{2\}, \{a,c\})$, and $(\emptyset,\{b\})$.
We are now in a position to define the substitution of $G$ and $H$
in $F$, also called partitional composition. Let $F(X,Y)$,
$G(X,Y)$, and $H(X,Y)$ be two-sort species. Then
\[
  F(G,H) \,= \!\!\! \bigsqcup_{\substack{
    \pi\,\in\,\Par[U,V] \\
    \kappa:\,\pi\,\to\,\{X,Y\}
    }}\!\! F[\beta,\gamma]
       \times G[B_1]\times\dots\times G[B_k]
       \times H[C_1]\times\dots\times H[C_\ell],
\]
in which $\beta = \{B_1,\dots,B_k\} = \kappa^{-1}(X)$ and
$\gamma = \{C_1,\dots,C_\ell\} = \kappa^{-1}(Y)$. Here we can think of
$\kappa$ as assigning one of two colors to each block of $\pi$; the
partition $\beta$ consists of the $X$-colored blocks, and $\gamma$
consists of the $Y$-colored blocks. Intuitively, an
$F(G,H)$-structure is obtained from an $F$-structure by inflating
each element of sort $X$ into a $G$-structure and each element of
sort $Y$ into an $H$-structure.

Let us now consider the species $\Phi(X,Y)$ whose structures on
$(U,V)$ are all functions $f:U\to V$, and whose transport of
structure is defined by
$\Phi[\sigma, \tau](f)=\tau \circ f\circ\sigma^{-1}$.
We shall prove the combinatorial equality $\Phi(X,Y)=E(E(X)Y)$. Let
$\Psi(X,Y)=E(E(X)Y)$ denote the right-hand side and define
$\alpha_{UV}:\Phi[U,V]\to \Psi[U,V]$ by
\[\alpha_{UV}(f) = \big\{\big(f^{-1}(v), v\big): v\in V \big\}.
\]
To show that $\Phi(X,Y)=\Psi(X,Y)$ then amounts to proving that, for any pair
of bijections $\sigma:U\to U'$ and $\tau:V\to V'$, the diagram
\[
  \begin{tikzpicture}[semithick]
    \matrix (m) [matrix of math nodes, row sep=4em, column sep=6em]
    { \Phi[U,V] & \Psi[U,V] \\
      \Phi[U',V'] & \Psi[U',V'] \\ };
    \path[->,font=\scriptsize]
      (m-1-1) edge node[auto] {$\alpha_{UV}$} (m-1-2)
              edge node[left] {$\Phi[\sigma,\tau]$} (m-2-1)
      (m-2-1) edge node[auto] {$\alpha_{U'V'}$} (m-2-2)
      (m-1-2) edge node[auto] {$\Psi[\sigma,\tau]$} (m-2-2);
  \end{tikzpicture}
\]
commutes. By direct calculation we have
\begin{align*}
  \Psi[\sigma,\tau](\alpha_{UV}(f))
    &= \big\{\big((\sigma\circ f^{-1})(v), \tau(v)\big): v\in V \big\};\\
  \alpha_{U'V'}(\Phi[\sigma,\tau](f))
    &= \big\{\big((\sigma\circ f^{-1}\circ \tau^{-1})(v), v\big): v\in V' \big\}.
\end{align*}
Since $\tau$ is a bijection these two sets are clearly equal.

Let $F(X,Y)$ be a two-sort species. Two $F$-structures
$s\in F[U,V]$ and $t\in F[U',V']$ are said to have the same
\emph{isomorphism type}, and one writes $s \sim t$, if there are
bijections $\sigma: U\to U'$ and $\tau: V\to V'$ such that
$F[\sigma,\tau](s) = t$. The \emph{type generating series} of $F(X,Y)$ is
\[
  \typ{F}(x,y) = \sum_{n,k\geq 0} \bigl|F[n,k]/{\sim} \bigr|\, x^n y^k,
\]
where $F[n,k]/{\sim}$ denotes the set of equivalence classes of
$F[n,k]$ with respect to ${\sim}$.

Two $F$-structures $s\in F[U,V]$ and $t\in F[U',V']$ are said to
have the same \emph{isomorphism type according to the sort $X$}, and
one writes $s \sim_X t$, if $V=V'$ and there is a bijection
$\sigma:U\to U'$ such that $F[\sigma, \Id_V](s) = t$. Similarly,
$s$ and $t$ are said to have the same \emph{isomorphism type
according to the sort $Y$}, and one writes $s \sim_Y t$, if $U=U'$
and there is a bijection $\tau:V\to V'$ such that
$F[\Id_U,\tau](s) = t$. It's easy to see that
${\sim_X}\circ{\sim_Y}={\sim}={\sim_Y}\circ{\sim_X}$, in which the circle denotes
composition of relations, and thus ${\sim_X}$ and ${\sim_Y}$ refine (factor) ${\sim}$ in
a natural way.
Define $\typeX{s} = \{ t : t\sim_X s \}$ and
$\typeY{s} = \{ t : t\sim_Y s \}$ as the equivalence classes of $s$
with respect to $\sim_X$ and $\sim_Y$.

As an example, let us revisit the species $\Phi(X,Y)$ of
functions $f:U\to V$. For $f\in \Phi[U,V]$ and $g\in \Phi[U',V]$ we have
\begin{alignat*}{2}
  g\sim_X f
  &\iff \Phi[\sigma, \Id_V](g) = f &\quad&\text{ for some bijection } \sigma:U\to U' \\
  &\iff g = f\circ\sigma           &\quad&\text{ for some bijection } \sigma:U\to U'.
\end{alignat*}
To be even more concrete, let $U=V=\{1,2,3\}$ and represent
$f\in \Phi[U,V]$ by $f(1)f(2)f(3)$. As we have just seen, $g\sim_X f$ in
$\Phi[U,V]$ precisely when $f$ and $g$ are equal up to a permutation of $U$. Thus
the set of equivalence classes of $\Phi[U,V]$ with respect to $\sim_X$
is
$$
\{
\typeX{111}, \typeX{112}, \typeX{113}, \typeX{122}, \typeX{123}, \\
\typeX{133}, \typeX{222}, \typeX{223}, \typeX{233}, \typeX{333}
\},
$$
where we have chosen the lexicographically smallest representative for
each equivalence class. Why are there 10 equivalence classes in this
case? Well, up to a permutation of $U$, a function from $U$ to
$V=\{1,2,3\}$ is determined by the number of elements, $x_i$, that $f$
maps to $i\in V$. Every element of $U$ has to be mapped to exactly one
element of $V$ and thus $x_1+x_2+x_3=|U|=3$. From elementary
combinatorics we know that the number of nonnegative solutions to this
equation is $\binom{|U|+|V|-1}{|U|}=\binom{5}{3}=10$.

For $f\in \Phi[U,V]$ and $g\in \Phi[U,V']$, we have
\begin{alignat*}{2}
  g\sim_Y f
  &\iff \Phi[\Id_U,\tau](g) = f &\quad&\text{ for some bijection } \tau:V\to V' \\
  &\iff \tau\circ g = f         &\quad&\text{ for some bijection } \tau:V\to V'.
\end{alignat*}
In particular, if $V=V'$ then $f\sim_Y g$ precisely when $f$ and $g$ are
equal up to a permutation of $V$. Now, up to a permutation of $V$, a
function $f\in \Phi[U,V]$ is determined by the set of its nonempty
fibres $\{ f^{-1}(v): v\in f(U)\}$, sometimes called the coimage of
$f$. This set forms a partition of $U$ into at most $|V|$ blocks and hence the
number of equivalence classes of $\Phi[U,V]$ with respect to $\sim_Y$ is
$S(n,0)+S(n,1)+\cdots+S(n,k)$, where $n=|U|$ and $k=|V|$.  If
$U=V=\{1,2,3\}$, then the $S(3,1)+S(3,2)+S(3,3)=1+3+1=5$ equivalence classes
are $\typeY{111}$, $\typeY{112}$, $\typeY{121}$, $\typeY{122}$, and $\typeY{123}$.

We now continue our development of the general framework.
Following Joyal~\cite{Joyal1986} we shall use the notation
\begin{align*}
  \simsum_U F[U,V] &\,=\,
    \big\{\, \typeX{s} : s\in F[U,V]\text{ for some finite set $U$}\, \big\}; \\
  \simsum_V F[U,V] &\,=\,
    \big\{\, \typeY{s} : s\in F[U,V]\text{ for some finite set $V$}\, \big\}.
\end{align*}
We would like to define a unisort ``species of types'' whose
structures are precisely the elements of $\simsum_UF[U,V]$. This
doesn't, however, quite fit our current species framework, the reason
being that the set $\simsum_UF[U,V]$ could be infinite. For
instance, consider (again) the species $\Phi(X,Y)$ of functions
$f:U\to V$. For any nonempty finite sets $U$ and $V$, both
$\simsum_U\Phi[U,V]$ and $\simsum_V\Phi[U,V]$ are infinite.
Indeed, if for simplicity we assume that $V=\{v\}$ is a singleton,
then $\simsum_U\Phi[U,V]$ contains one element per positive size of
the domain:
\[
  \begin{tikzpicture}[semithick,baseline=(x.base),xscale=0.65, yscale=0.85]
    \draw[thin, rounded corners] (-0.7, 0.5) rectangle ++(3.4,-1.3);
    \draw[rounded corners] (-0.4, 0.3) rectangle ++(0.8,-0.9);
    \draw[rounded corners] ( 1.6, 0.3) rectangle ++(0.8,-0.9);
    \node[fill, circle, inner sep=1.4pt] (x) at (0,0) {};
    \node[draw, circle, inner sep=1.4pt, label=below:$\scriptstyle v$] (y) at (2, 0) {};
    \draw[postaction={decorate},decoration={markings,mark=at position 0.5 with {\arrow{>}}}] (x) -- (y);
  \end{tikzpicture}
  \qquad
  \begin{tikzpicture}[semithick,baseline=(x.base),xscale=0.65, yscale=0.85]
    \draw[thin, rounded corners] (-0.7, 0.5) rectangle ++(3.4,-1.5);
    \draw[rounded corners] (-0.4, 0.3) rectangle ++(0.8,-1.1);
    \draw[rounded corners] ( 1.6, 0.3) rectangle ++(0.8,-1.1);
    \node[fill, circle, inner sep=1.4pt] (x1) at (0,0) {};
    \node[fill, circle, inner sep=1.4pt] (x2) at (0,-0.4) {};
    \node[draw, circle, inner sep=1.4pt, label=below:$\scriptstyle v$] (y) at (2, 0) {};
    \draw[postaction={decorate},decoration={markings,mark=at position 0.5 with {\arrow{>}}}] (x1) -- (y);
    \draw[postaction={decorate},decoration={markings,mark=at position 0.5 with {\arrow{>}}}] (x2) -- (y);
  \end{tikzpicture}
  \qquad
  \begin{tikzpicture}[semithick,baseline=(x.base),xscale=0.65, yscale=0.85]
    \draw[thin, rounded corners] (-0.7, 0.5) rectangle ++(3.4,-1.65);
    \draw[rounded corners] (-0.4, 0.3) rectangle ++(0.8,-1.25);
    \draw[rounded corners] ( 1.6, 0.3) rectangle ++(0.8,-1.25);
    \node[fill, circle, inner sep=1.4pt] (x1) at (0,0) {};
    \node[fill, circle, inner sep=1.4pt] (x2) at (0,-0.3) {};
    \node[fill, circle, inner sep=1.4pt] (x3) at (0,-0.6) {};
    \node[draw, circle, inner sep=1.4pt, label=below:$\scriptstyle v$] (y) at (2, 0) {};
    \draw[postaction={decorate},decoration={markings,mark=at position 0.5 with {\arrow{>}}}] (x1) -- (y);
    \draw[postaction={decorate},decoration={markings,mark=at position 0.5 with {\arrow{>}}}] (x2) -- (y);
    \draw[postaction={decorate},decoration={markings,mark=at position 0.5 with {\arrow{>}}}] (x3) -- (y);
  \end{tikzpicture}
  \qquad\raisebox{-2.5pt}{$\boldsymbol{\cdots}$}
\]
Similarly, if $U=\{u\}$ is a singleton, then
$\simsum_V\Phi[U,V]$ contains one element per positive size of the
codomain:
\[
  \begin{tikzpicture}[semithick,baseline=(x.base),xscale=0.65, yscale=0.85]
    \draw[thin, rounded corners] (-0.7, 0.5) rectangle ++(3.4,-1.3);
    \draw[rounded corners] (-0.4, 0.3) rectangle ++(0.8,-0.9);
    \draw[rounded corners] ( 1.6, 0.3) rectangle ++(0.8,-0.9);
    \node[fill, circle, inner sep=1.4pt, label=below:$\scriptstyle u$] (x) at (0,0) {};
    \node[draw, circle, inner sep=1.4pt] (y) at (2, 0) {};
    \draw[postaction={decorate},decoration={markings,mark=at position 0.5 with {\arrow{>}}}] (x) -- (y);
  \end{tikzpicture}
  \qquad
  \begin{tikzpicture}[semithick,baseline=(x.base),xscale=0.65, yscale=0.85]
    \draw[thin, rounded corners] (-0.7, 0.5) rectangle ++(3.4,-1.5);
    \draw[rounded corners] (-0.4, 0.3) rectangle ++(0.8,-1.1);
    \draw[rounded corners] ( 1.6, 0.3) rectangle ++(0.8,-1.1);
    \node[fill, circle, inner sep=1.4pt, label=below:$\scriptstyle u$] (x) at (0,0) {};
    \node[draw, circle, inner sep=1.4pt] (y) at (2, 0) {};
    \node[draw, circle, inner sep=1.4pt]     at (2,-0.4) {};
    \draw[postaction={decorate},decoration={markings,mark=at position 0.5 with {\arrow{>}}}] (x) -- (y);
  \end{tikzpicture}
  \qquad
  \begin{tikzpicture}[semithick,baseline=(x.base),xscale=0.65, yscale=0.85]
    \draw[thin, rounded corners] (-0.7, 0.5) rectangle ++(3.4,-1.65);
    \draw[rounded corners] (-0.4, 0.3) rectangle ++(0.8,-1.25);
    \draw[rounded corners] ( 1.6, 0.3) rectangle ++(0.8,-1.25);
    \node[fill, circle, inner sep=1.4pt, label=below:$\scriptstyle u$] (x) at (0,0) {};
    \node[draw, circle, inner sep=1.4pt] (y) at (2, 0) {};
    \node[draw, circle, inner sep=1.4pt]     at (2,-0.3) {};
    \node[draw, circle, inner sep=1.4pt]     at (2,-0.6) {};
    \draw[postaction={decorate},decoration={markings,mark=at position 0.5 with {\arrow{>}}}] (x) -- (y);
  \end{tikzpicture}
  \qquad\raisebox{-2.5pt}{$\boldsymbol{\cdots}$}
\]
We'll handle this using weighted species, but first we need to
generalize the notion of an $R$-weighted species. Recall that we
defined an $R$-weighted set to be a pair $(A,w)$, where $A$ is a
finite set and $w:A\to R$. Let us say that $(A,w)$ is a
\emph{summable $R$-weighted set} if $A$ is a (possibly infinite) set,
$w:A\to R$, and the total weight $|A|_w = \sum_{a\in A}w(a)$ exists.
Then an $R$-weighted species is a functor
$F:\CoreFinSet\to(\Set/R)$, where $(\Set/R)$ is the category of
summable $R$-weighted sets with morphisms of $R$-weighted sets as
arrows.

We are now ready to define the desired species of types. For any two-sort
species $F(X,Y)$ we define the $\ZZ[[x]]$-weighted species
$\simsum_UF[U,Y)$ as follows: For any finite set $V$, let the
collection of structures on $V$ be the summable $\ZZ[[x]]$-weighted
set $\bigl(\simsum_UF[U,V], w \bigr)$,
where for any $s$ in $F[U, V]$ the weight of $\typeX{s}$ in $\simsum_UF[U,V]$ is $w(s)=x^{|U|}$.
For any bijection $\tau:V\to V'$, define the
transport of structure $\simsum_UF[U,\tau]$ as the map
\[
 [\,s\,]_X \,\longmapsto\,
 \bigl[\,F[\Id_U,\tau](s)\,\bigr]_X\hspace{1pt},
\]
which is clearly weight-preserving. The resulting species,
$\simsum_UF[U,Y)$, is called the \emph{species of types according to
$X$}; the species $\simsum_VF(X,V]$ of types according to $Y$ is
defined analogously. We'll use the following notation for their generating series:
\begin{align*}
  \simsum_U F[U,y)
  &= \sum_{n,k\geq 0}\big|\,F[n,k]/{\sim}_X\,\big|\,x^n\frac{y^k}{k!},\\
  \simsum_V F(x,V]
  &= \sum_{n,k\geq 0}\big|\,F[n,k]/{\sim}_Y\,\big|\,\frac{x^n}{n!}y^k,
\end{align*}
where $F[U,V]/{\sim}_X = \{\typeX{s}: s\in F[U,V] \}$ and
$F[U,V]/{\sim}_Y = \{\typeY{s}: s\in F[U,V] \}$ are the sets of
equivalence classes with respect to the relations $\sim_X$ and
$\sim_Y$. Let us briefly return to the $\Phi(X,Y)$ species for an example. We
discussed the equivalence classes of $\Phi[U,V]$ with respect to
$\sim_X$ and $\sim_Y$ in some detail above, and from that discussion it
is clear that
\begin{align*}
  \simsum_U \Phi[U,y)
       &= \sum_{n,k\geq 0}\binom{n+k-1}{n}\,x^n\frac{y^k}{k!}; \\
  \simsum_V \Phi(x,V]
      &= \sum_{n,k\geq 0}\big(S(n,0)+\cdots+S(n,k)\big)\,\frac{x^n}{n!}y^k.
\end{align*}

The cycle index series for a two-sort species $F(X,Y)$ is defined by
\begin{equation*}
  Z_{F}(x_1,x_2,\dots;y_1,y_2,\dots) =
    \sum_{\substack{n\geq 0\\ k\geq 0}}\,\frac{1}{n!k!}
    \sum_{\substack{\sigma\in\Sym[n]\\ \tau\in\Sym[k]}}|\Fix F[\sigma,\tau]|\,
      x_1^{c_1(\sigma)}x_2^{c_2(\sigma)}\cdots\, y_1^{c_1(\tau)}y_2^{c_2(\tau)}\cdots
\end{equation*} One can show that if $F(X,Y)$, $G(X,Y)$, and
$H(X,Y)$ are two-sort species, then $Z_{F+G}=Z_F+Z_G$,
$Z_{FG}=Z_F Z_G$, and
\begin{multline*}
  \ \quad Z_{F(G,H)}(\,x_1,x_2,x_3\dots;\, y_1,y_2,y_3\dots\,) = \\
      \arraycolsep=0.9pt
      \renewcommand{\arraystretch}{1.1}
      \begin{array}{rllllllllllllllllll}
        Z_F\big(\,Z_G( &x_1,&x_2,&x_3&\dots),\, &Z_G(x_2,&x_4,&x_6,\dots),\, &Z_G(x_3,&x_6,&x_9,&\dots), \dots ; \\
                  Z_H( &y_1,&y_2,&y_3&\dots),\, &Z_H(y_2,&y_4,&y_6,\dots),\, &Z_H(y_3,&y_6,&y_9,&\dots), \dots\,\big)
      \end{array}\quad
\end{multline*}
The major generating series associated with a two-sort
species relate to the cycle index series as follows:
\[
\arraycolsep=1.3pt
\renewcommand{\arraystretch}{1.3}
\begin{array}{rclllllllllll}
               F(x,y) &\,=\,& Z_F(&x, 0, 0, 0,  &\dots&;\, &y, 0, 0, 0,  &\dots &); \\
     \simsum_U F[U,y) &\,=\,& Z_F(&x, x^2, x^3, &\dots&;\, &y, 0, 0, 0,  &\dots &); \\
     \simsum_V\hspace{0.9pt}F(x,V] &\,=\,& Z_F(&x, 0, 0, 0,  &\dots&;\, &y, y^2, y^3, &\dots &); \\
         \typ{F}(x,y) &\,=\,& Z_F(&x, x^2, x^3, &\dots&;\, &y, y^2, y^3, &\dots &).
\end{array}
\]

\hypertarget{the-twelvefold-way}{%
\section{The twelvefold way}\label{the-twelvefold-way}}

Let $R(X)$ be a unisort species and consider the two-sort species
\[
\Fun(X,Y) = E(R(X) Y)
\]
of \emph{functions with $R$-enriched fibres}~\cite[p.\ 113]{Bergeron1998}:\smallskip
\[
  \begin{tikzpicture}[semithick, scale=0.85]
    \draw[rounded corners] (0,0) rectangle +(2,4);
    \draw [zigzag] (0,1.00) -- (2,1.25);
    \draw [zigzag] (0,2.80) -- (2,2.70);
    \draw[rounded corners] (3.5,0) rectangle +(1.5,4);
    \node (R1) at (0.75,0.5) {\small $R(X)$};
    \node (R2) at (0.75,1.8) {\small $R(X)$};
    \node (R3) at (0.75,3.3) {\small $R(X)$};
    \node[draw, circle, inner sep=1.4pt] (Y1) at (4.1,0.7) {};
    \node at (4.5,0.7) {\small $Y$};
    \node[draw, circle, inner sep=1.4pt] (Y2) at (4.1,1.95) {};
    \node at (4.5,1.9) {\small $Y$};
    \node[draw, circle, inner sep=1.4pt] (Y3) at (4.1,3.2) {};
    \node at (4.5,3.2) {\small $Y$};
    \node[fill, circle, inner sep=1.4pt] (r11) at (1.5,0.5) {};
    \draw[postaction={decorate},decoration={markings,mark=at position 0.5 with {\arrow{>}}}] (r11) -- (Y1);
    \node[fill, circle, inner sep=1.4pt] (r21) at (1.1,2.30) {};
    \node[fill, circle, inner sep=1.4pt] (r22) at (1.4,1.59) {};
    \node[fill, circle, inner sep=1.4pt] (r23) at (1.7,1.95) {};
    \draw[postaction={decorate},decoration={markings,mark=at position 0.5 with {\arrow{>}}}] (r21) -- (Y2);
    \draw[postaction={decorate},decoration={markings,mark=at position 0.5 with {\arrow{>}}}] (r22) -- (Y2);
    \draw[postaction={decorate},decoration={markings,mark=at position 0.5 with {\arrow{>}}}] (r23) -- (Y2);
    \node[fill, circle, inner sep=1.4pt] (r31) at (1.5,3.55) {};
    \node[fill, circle, inner sep=1.4pt] (r32) at (1.6,3.10) {};
    \draw[postaction={decorate},decoration={markings,mark=at position 0.5 with {\arrow{>}}}] (r31) -- (Y3);
    \draw[postaction={decorate},decoration={markings,mark=at position 0.5 with {\arrow{>}}}] (r32) -- (Y3);
  \end{tikzpicture}
\]
Note that if $R(X)=E(X)$ is the set species, then $\Fun(X,Y)=\Phi(X,Y)$ is
the species of all functions $f:U\to V$ from above. If $R(X)=E_+(X)$ is the
species of nonempty sets, then we are requiring that each fibre
$f^{-1}(v) = \{u\in U: f(u) = v\}$ is nonempty, and so $f$ is
surjective. If $R(X)=1+X$, then each fibre is empty or a singleton, and we
get injective functions.

\begin{lemma}\label{lemma:funs}
  Let $\Fun(X,Y) = E(R(X)Y)$ be as above. Then
  \begin{align*}
    \vphantom{\vbox to 0.1em{}} \Fun(x,y) &= \exp\bigl( yR(x) \bigr); \\
    \vphantom{\vbox to 1.6em{}}     \xFun &= \exp\bigl(y\typ{R}(x)\bigr); \\
    \vphantom{\vbox to 1.6em{}}     \yFun &= \exp\Bigl( y\bigl(R(x)-R(0)\bigr)\Bigr)\bigl(1-y\bigr)^{-R(0)}; \\
    \vphantom{\vbox to 1.6em{}}\uFun(x,y) &= \exp\sum_{k\geq 1} \frac{y^k}{k} \typ{R}(x^k).
  \end{align*}
\end{lemma}
\begin{proof}
  The first equality is immediate. Calculating the cycle index series we find that
  \begin{equation}\label{eq:Z_Phi_R}
    Z_{\Fun}(x_1,x_2,x_3,\dots;y_1,y_2,y_3,\dots) =
    \exp\sum_{k\geq 1}\frac{y_k}{k} Z_R\bigl(x_k,x_{2k},x_{3k},\dots\bigr)
  \end{equation}
  and the remaining three identities follow from calculating
  $Z_{\Fun}(x,x^2,x^3,\dots;y,0,0,0,\dots)$,
  $Z_{\Fun}(x,0,0,0,\dots;y,y^2,y^3,\dots)$, and
  $Z_{\Fun}(x,x^2,x^3,\dots;y,y^2,y^3,\dots)$.  Let us do the second
  of these three calculations. By isolating the first term in the sum on
  the right-hand side of \eqref{eq:Z_Phi_R} it follows that
  \begin{align*}
    \yFun
    &= Z_{\Fun}(x,0,0,0,\dots;y,y^2,y^3,\dots) \\
    &= \exp\biggl(
      yZ_R(x,0,0,\dots) +
      \sum_{k\geq 2}\frac{y^k}{k} Z_R\bigl(0,0,0,\dots\bigr)
      \biggr).
  \end{align*}
  Since $R(x) = Z_R(x,0,0,\dots)$ this can be further simplified as
  follows:
  \begin{align*}
    \yFun
    &= \exp\biggl( yR(x) + R(0)\sum_{k\geq 2}\frac{y^k}{k} \biggr) \\
    &= \exp\Bigl( yR(x) - yR(0) + R(0)\log (1-y)^{-1}\Bigr) \\
    &= \exp\Bigl( y\bigl(R(x)-R(0)\bigr)\Bigr)\bigl(1-y\bigr)^{-R(0)}.
  \end{align*}
  We leave verification of the remaining two identities to the reader.
\end{proof}

As an example, let us apply the above lemma when $R(X)=X$. Then each fibre
is a singleton and $\Fun[X]$ is the species of bijections $f:U\to V$. The
four generating series are
\begin{align}
   \Fun[X](x,y) &= e^{xy} = \sum_{n\geq 0}n!\cdot \frac{x^ny^n}{(n!)^2};\label{1:F[X]}\\
  \xFun[X] &= e^{xy} = \sum_{n\geq 0}1\cdot x^n\frac{y^n}{n!};\label{2:F[X]}\\
  \yFun[X] &= e^{xy} = \sum_{n\geq 0}1\cdot \frac{x^n}{n!}y^n;\label{3:F[X]}\\
  \uFun[X](x,y) &= \exp\sum_{k\geq 1} \frac{x^ky^k}{k}
   = \frac{1}{1-xy} = \sum_{n\geq 0}1\cdot x^ny^n\label{4:F[X]}.
\end{align}
In other words, to distribute $n$ balls into $k$ urns so that each urn
gets exactly one ball is impossible unless there are as many urns as
there are balls, that is, unless $n=k$. And, if $n=k$, then there are $n!$ ways to
distribute the balls if balls and urns are distinct,
and there is exactly one way if the balls or the urns are identical.

\begin{lemma}
  Let $A(X)$ and $B(X)$ be two unisort species, and let
  the two-sort species $\Fun(X,Y)=E(R(X)Y)$ be as above.  Then
  $$\Fun[A+B](X,Y) = \Fun[A](X,Y)\Fun[B](X,Y).
  $$
  In particular, if $R(X)=1+R_+(X)$, where $R_+(X)$ is the species of nonempty
  $R$-structures, then
  $$\Fun(X,Y) = E(Y)\Fun[R_+](X,Y).
  $$
\end{lemma}
\begin{proof}
  For the first statement we have
  \[
    \Fun[A+B]
      = E\bigl((A+B)(X)Y \bigr)
      = E\bigl(A(X)Y +B(X)Y \bigr)
      = E\bigl(A(X)Y \bigr)E\bigl(B(X)Y \bigr)
      = \Fun[A]\Fun[B].
  \]
  The second statement is simply $\Fun[1](X,Y)=E(Y)$.
\end{proof}

\begin{corollary}\label{cor:zfun}
  Let $A(X)$ and $B(X)$ be two unisort species, and let $\Fun(X,Y)$
  be as above.  Then $Z_{\Fun[A+B]} = Z_{\Fun[A]}Z_{\Fun[B]}$.  In
  particular, if $R(X)=1+R_+(X)$ then $Z_{\Fun} = Z_{E(Y)}Z_{\Fun[R_+]}\!$,
  and consequently\vspace{-1.5ex}
  \begin{align*}
    \vphantom{\vbox to 0.1em{}} \Fun(x,y) &= E(y)\cdot\Fun[R_+](x,y), \\
    \vphantom{\vbox to 1.4em{}}\xFun &= E(y)\cdot\xFun[R_+], \\
    \vphantom{\vbox to 1.4em{}}\yFun &= \typ{E}(y)\cdot\yFun[R_+], \\
    \vphantom{\vbox to 1.4em{}}\uFun(x,y) &= \typ{E}(y)\cdot\uFun[R_+](x,y),
  \end{align*}
  in which $E(y)=e^y$ and $\typ{E}(y) = 1/(1-y)$.
\end{corollary}


The twelvefold way can be seen as the twelve generating series obtained
from calculating, for each $R(X)\in \{E(X), E_+(X), 1+X\}$, the four generating
series in Lemma~\ref{lemma:funs}. Note that it suffices to know $R(x)$ and
$\typ{R}(x)$ to calculate those four series. For reference we record $R(x)$
and $\typ{R}(x)$ in Table~\ref{table:R}.
\begin{table}[h]
  \[
  \arraycolsep=1em
  \renewcommand{\arraystretch}{1.3}
  \begin{array}{*3{>{\displaystyle}r|}}
    R(X) &R(x) &\typ{R}(x) \\
    \hline
    X      &x           &x \\
    1+X    &1+x         &1+x\\
    E_+(X) &e^x-1       & x(1-x)^{-1} \\
    E(X)   & e^x        & (1-x)^{-1} \\
    L_+(X) & x(1-x)^{-1} & x(1-x)^{-1} \\
    L(X)   & (1-x)^{-1} & (1-x)^{-1} \\
    \Cyc(X)& \log\,(1-x)^{-1} & x(1-x)^{-1}
  \end{array}
  \]
  \caption{Some (type) generating series}\label{table:R}
\end{table}
We shall run through the cases of this table going from top to bottom.
Having already done $R(X)=X$, we proceed with $R(X)=1+X$. By direct
calculation, or from Corollary~\ref{cor:zfun} and equations
\eqref{1:F[X]} through \eqref{4:F[X]}, we immediately get
\begin{align}
  \Fun[1+X](x,y) &= e^{y(1+x)}
   = \sum_{n,k\geq 0}(k)_n\,\frac{x^n}{n!}\frac{y^k}{k!},\label{1:F[1+X]}\\
  \xFun[1+X] &= e^{y(1+x)}
   = \sum_{n,k\geq 0}\binom{k}{n}\,x^n\frac{y^k}{k!},\label{2:F[1+X]}\\
  \yFun[1+X] &= \frac{1}{1-y}e^{xy}
   = \sum_{n,k\geq 0}[n\leq k]\,\frac{x^n}{n!}y^k,\label{3:F[1+X]}\\
  \uFun[1+X](x,y) &= \frac{1}{1-y}\cdot\frac{1}{1-xy}
   = \sum_{n,k\geq 0}[n\leq k]\,x^ny^k,\label{4:F[1+X]}
\end{align}
where $(k)_n=k(k-1)\cdots(k-n+1)$ denotes the
falling factorial and $[P]$ is the Iverson bracket, which is 1 if the
proposition $P$ is true and 0 if it is false.

Assuming that each urn contains at most one ball, there are thus $(k)_n$
ways to distribute $n$ distinct balls into $k$ distinct urns; there are
$\binom{k}{n}$ ways to distribute $n$ identical balls into $k$ distinct
urns; if $n\leq k$, then there's exactly one way to distribute $n$ balls
(distinct or identical) into $k$ identical urns, and if $n>k$ there's no
way to distribute the balls.

As a side note, the generating series in \eqref{2:F[1+X]} is the same as
for the $\ZZ[x]$-weighted species $\Pow_w(Y)=E_w(Y)E(Y)$ in which the
weight of an $E_w$-structure is $w(A)=x^{|A|}$. Indeed, those
species are combinatorially equal:
$\tXFun[1+X]=\Pow_w(Y)$.

Let us now consider $R(X)=E_+(X)$. Using Lemma~\ref{lemma:funs} we find that
\begin{align}
  \Fun[E_+](x,y)
  &= \exp\big( y(e^x - 1) \big)
    = \sum_{n,k\geq 0}k! S(n,k)\,\frac{y^k}{k!}\frac{x^n}{n!},\label{1:F[E_+]}\\
  \xFun[E_+]
  &= \exp\left( \frac{xy}{1-x}\right)
    = \sum_{n,k\geq 0}\binom{n-1}{n-k}\, x^n\frac{y^k}{k!},\label{2:F[E_+]}\\
  \yFun[E_+]
  &= \exp\Big( y(e^x - 1) \Big)
    = \Par_w(x) = \sum_{n,k\geq 0}S(n,k)\,\frac{x^n}{n!}y^k,\label{3:F[E_+]}\\
  \uFun[E_+](x,y)
  &= \exp \sum_{k\geq 1}\frac{y^k}{k}\frac{x^k}{1-x^k}
    = \prod_{k\geq 1}\frac{1}{1-yx^k}
    = \sum_{n,k\geq 0}p_k(n)\,x^ny^k,\label{4:F[E_+]}
\end{align}
where $S(n,k)$ is a Stirling number of the second kind and $p_k(n)$ is
the number of integer partitions of $n$ with $k$ parts. Note that the
second equality in \eqref{4:F[E_+]} is the same as
\eqref{eq:integer-partitions-with-k-blocks} on page
\pageref{eq:integer-partitions-with-k-blocks}, above. The interpretation
of the formulas above in terms of balls and urns should at this stage be
clear.

The last case of the twelvefold way is $R(X)=E(X)$. Since $E(X)=1+E_+(X)$ we may use
Corollary~\ref{cor:zfun} when convenient, and we have
\begin{align}
  \Fun[E](x,y)
      &= \exp\big( y e^x  \big)
       = \sum_{n,k\geq 0}k^n\,\frac{x^n}{n!}\frac{y^k}{k!}; \\
  \xFun[E]
       &= \exp\left(\frac{y}{1-x}\right)
       = \sum_{n,k\geq 0}\binom{n+k-1}{n}\,x^n\frac{y^k}{k!}; \\
  \yFun[E]
      &= \frac{1}{1-y}\yFun[E_+]
      = \sum_{n,k\geq 0}\big(S(n,0)+\cdots+S(n,k)\big)\,\frac{x^n}{n!}y^k; \\
  \uFun[E](x,y)
      &= \frac{1}{1-y}\uFun[E_+](x,y)
        =\sum_{n,k\geq 0}\big(p_0(n) +\cdots+ p_k(n) \big)\,x^ny^k.
\end{align}

\begin{table}
  {\small
  $$
  \arraycolsep=0.75em
  \renewcommand{\arraystretch}{2.1}
  \begin{array}{*5{>{\displaystyle}r|}}
    R(X) &\Fun(x,y) &\xFun &\yFun &\uFun(x,y) \\
    \hline
    X       &\exp(xy)
            &\exp(xy)
            &\exp(xy)
            &\frac{1}{1-xy} \\
    1+X     &\exp(y(1+x))
            &\exp(y(1+x))
            &\frac{1}{1-y}\exp(xy)
            &\frac{1}{1-y}\cdot\frac{1}{1-xy} \\
    E_+(X)  &\exp(y(e^x-1))
            &\exp\left(\frac{xy}{1-x}\right)
            &\exp(y(e^x-1))
            &\prod_{k\geq 1}\frac{1}{1-yx^k} \\
    E(X)    & \exp\bigl(ye^x\bigr)
            &\exp\left(\frac{y}{1-x}\right)
            & \frac{1}{1-y}\exp\big(y(e^x-1)\big)
            & \frac{1}{1-y}\prod_{k\geq 1}\frac{1}{1-yx^k} \\
    L_+(X)  & \exp\left(\frac{xy}{1-x}\right)
            & \parbox[t]{7.8em}{\raggedleft See $R(X)=E_+(X)$}
            & \exp\left(\frac{xy}{1-x}\right)
            & \parbox[t][1.9em]{7.8em}{\raggedleft See $R(X)=E_+(X)$}\\
    L(X)    & \exp\left(\frac{y}{1-x}\right)
            & \parbox[t]{7.8em}{\raggedleft See $R(X)=E(X)$}
            & \frac{1}{1-y}\exp\left(\frac{xy}{1-x}\right)
            & \parbox[t]{7.8em}{\raggedleft See $R(X)=E(X)$} \\
    \Cyc(X) & (1-x)^{-y}
            & \parbox[t]{7.8em}{\raggedleft See $R(X)=E_+(X)$}
            & (1-x)^{-y}
            & \parbox[t]{7.8em}{\raggedleft See $R(X)=E_+(X)$}
  \end{array}
  $$
  }
  \caption{The twenty-two-fold way}\label{table:manifold}
\end{table}

Bogart's~\cite{Bogart2004} twentyfold way is a natural extension of
Rota's twelvefold way, which, in our setting, amounts to having more
choices for $R(X)$. For the twelvefold way $R(X)$ is one of $E(X)$, $E_+(X)$, or
$1+X$. For the twentyfold way there are three additional options:
$X$, $L(X)$, and $L_+(X)$. At first glance this appears to give us
$6\cdot 4 = 24$ cases, but since $\typ{E}(x)=\typ{L}(x)$ and
$\typ{E}_+(x)=\typ{L}_+(x)$ there are only 20 distinct cases.

Let us explore these additional cases, and let us start with
$R(X)=L_+(X)$. This gives us two new series (the other two series being the
same as for $R(X)=E_+(X)$):
\begin{align*}
  \Fun[L_+](x,y) &= \exp\left(\frac{xy}{1-x}\right)
    = \sum_{n,k\geq 0}k!L(n,k)\,\frac{x^n}{n!}\frac{y^k}{k!}, \\
  \yFun[L_+] &= \exp\left(\frac{xy}{1-x}\right)
    = \sum_{n,k\geq 0}L(n,k)\,\frac{x^n}{n!}y^k,
\end{align*}
where $L(n,k)=\binom{n}{k}(n-1)_{n-k}$ denotes a Lah
number~\cite{Lah1955}. Using the language of species, the standard
combinatorial interpretation of these numbers is as follows: Define the
species $\Lah=E(L_+)$, so that a $\Lah$-structure is a set of nonempty
disjoint linear orders, such as $\{63,9,4815,27\}\in \Lah[9]$. The
number of $\Lah$-structures on $[n]$ with $k$ blocks / linear orders is
$L(n,k)$.

With $R(X)=L(X)$ the two new series are
\begin{align*}
   \Fun[L](x,y) &= \exp\left(\frac{y}{1-x}\right)
      = \sum_{n,k\geq 0}(n+k-1)_n\,\frac{x^n}{n!}\frac{y^k}{k!}, \\
  \yFun[L] &= \frac{1}{1-y}\yFun[L_+]
      = \sum_{n,k\geq 0}\bigl(L(n,0)+\cdots+L(n,k)\bigr)\frac{x^n}{n!}y^k,
\end{align*}
and this concludes the twentyfold way.

We have seen Stirling numbers of the second kind ($R(X)=E_+(X)$). We have also
seen $\Lah$ numbers ($R(X)=L_+(X)$), which are sometimes referred to as
Stirling numbers of the third kind~\cite{Tsylova1985}. But what about
the Stirling numbers of the first kind?  Is there a choice of $R(X)$ which
gives those numbers? Note that the species constructions for the
Stirling numbers of the second and third kind follow a similar pattern:
$\Par(X) = E(E_+(X))$ and $\Lah(X) = E(L_+(X))$. Assume, more generally, that the
species $R(X)$ and $R^c(X)$ are related by $R(X)=E(R^c(X))$, so that an
$R$-structure is a set of disjoint $R^c$-structures. We may then call
$R^c$ the species of \emph{connected} $R$-structures~\cite[p.\ 46]{Bergeron1998}.

\begin{proposition}\label{prop:components}
  Assume that $R(X)=E(R^c(X))$ and define the $\ZZ[y]$-weighted species
  $R_w(X)=E(yR^c(X))$, where $yR^c(X)$ is like $R^c(X)$ but with a weight $y$ given
  to each structure. Then
  $$
  \tYFun[R^c] = R_w(X).
  $$
\end{proposition}
\begin{proof}
  A typical $\tYFun[R^c]$-structure on a finite set $U$ is an
  equivalence class $[s]_Y$, where $s\in\Fun[R^c][U,V]$ for some finite
  set $V=\{v_1,\dots,v_k\}$. This representative, $s$, can in turn be
  written $s=\{(s_v, v): s_v\in R^c[B_v], v\in V\}$ for some partition
  $\{B_{v_1},\dots,B_{v_k}\}$ of $U$. Define
  $\alpha_U: \simsum_V\Fun[R^c][U,V] \to R_w[U]$ by
  $\alpha_U(s)=\{s_v: v\in V\}$. It's easy to see that this definition
  is independent of the representative for $[s]_Y$. It preserves weights
  too: the weight of $[s]_Y$ is $y^{|V|}=y^k$, which is also the weight
  of $\{s_v: v\in V\}$. Further, for any bijection $\sigma:U\to U'$, the
  relevant diagram commutes and thus $\alpha_U$ is the desired natural
  isomorphism.
\end{proof}

This proposition applies to every other row of
Table~\ref{table:manifold}: the connected components of a set, $E(X)$, are
the points, $X$; the connected components of a partition, $\Par(X)$, are
the blocks, $E_+(X)$; the connected components of a $\Lah$-structure are
the linear orders, $L_+(X)$; and the connected components of a permutation
are the cycles, $\Cyc(X)$. See also Figure~\ref{fig:balls-in-urns} for an
illustration of these cases. Letting
$E_w(X)=E(yX)$,
$\Par_w(X)=E(yE_+(X))$, $\Lah_w(X)=E(yL_+(X))$, and $\Sym_w(X)=E(y\Cyc(X))$, and applying
Proposition~\ref{prop:components} we have
$$
\arraycolsep=2pt
\renewcommand{\arraystretch}{1.5}
\begin{array}{rclcrcl}
  \tYFun[X] &=& E_w(X); && \tYFun[E_+] &=& \Par_w(X); \\
  \tYFun[L_+] &=& \Lah_w(X); && \tYFun[\Cyc] &=& \Sym_w(X).
\end{array}
$$
We could add to this list by considering other natural pairs $R(X)$ and
$R^c(X)$ such as graphs and connected graphs, or forests and trees.
\begin{figure}
  \[
    \includegraphics[width=0.55\textwidth]{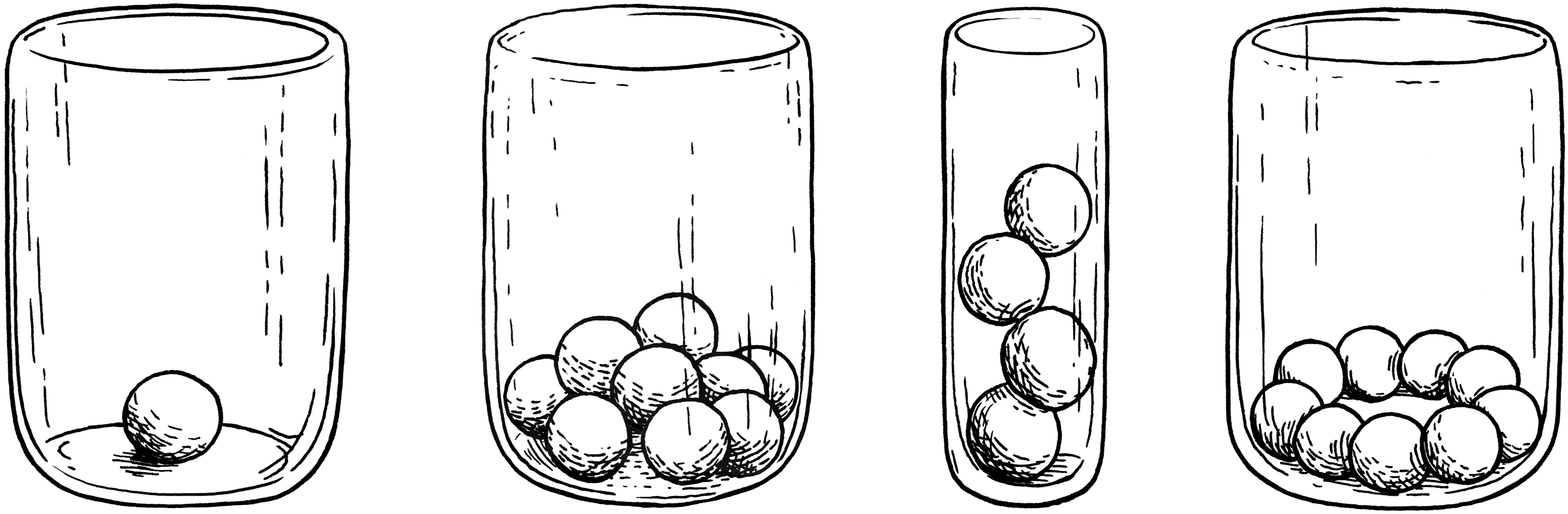}
  \]
  \caption{Urns illustrating $R(X)=X$, $E_+(X)$, $L_+(X)$, and $\Cyc(X)$}
  \label{fig:balls-in-urns}
\end{figure}

Returning to the Stirling numbers, we see from above that
$$
\yFun[\Cyc] = \Sym_w(x) = (1-x)^{-y}
    = \sum_{n,k\geq 0}c(n,k)\,\frac{x^n}{n!}y^k,
$$
where $c(n,k)$ denotes an unsigned Stirling number of the first kind. The other three
series for $R(X)=\Cyc(X)$ are $\xFun[\Cyc] =\xFun[E_+]$,
$\uFun[\Cyc](x,y) = \uFun[E_+](x,y)$, and
$\Fun[\Cyc](x,y) = (1-x)^{-y}= \sum_{n,k\geq 0}k!
c(n,k)\,\frac{x^n}{n!}\frac{y^k}{k!}$.  Each choice of $R(X)$ that we have
discussed up to this point can be found in Table~\ref{table:R}, and the
four series associated with that choice are listed in
Table~\ref{table:manifold}.

We will stop here, but note that there are, of course, other choices for
$R(X)$ that may be natural, but we have omitted. For instance, setting
$$R(X)=E_{\mathrm{odd}}(X) = E_1(X)+E_3(X)+\cdots
$$
would require that each urn contains an odd number of balls.


\bibliographystyle{abbrv}
\bibliography{bibliography.bib}

\bigskip
{\footnotesize
\textsc{Division of Mathematics, The Science Institute, University of Iceland}
}

\end{document}